\documentclass[11pt]{article}
\usepackage{amssymb}
\usepackage{amsmath}
\usepackage{graphicx}
\usepackage{color}
\usepackage{amsthm,amsmath,amssymb}
\usepackage{mathrsfs}
\usepackage{helvet}
\usepackage{graphicx}
\usepackage{url}
\usepackage{bm}
\usepackage{multirow}
\usepackage{booktabs}
\usepackage{algorithm}
\usepackage{algorithmic}
\usepackage{graphicx}
\usepackage{subfigure}
\usepackage{float}
\usepackage{makecell}
\usepackage{hyperref} 
\usepackage{indentfirst}
\makeatletter

\newcommand{\Rmnum}[1]{\expandafter@slowromancap\romannumeral #1@}
\makeatother

\usepackage{subcaption}
\graphicspath{{1107figures/}}
\def\dfrac{\displaystyle\frac}

\newtheorem{theorem}{Theorem}[section]                   
\newtheorem{definition}{Definition}[section]                   

\newtheorem{lemma}[theorem]{Lemma}

\newtheorem{remark}{Remark}[section]

\newtheorem{assumption}{Assumption}
\numberwithin{equation}{section}

\begin{document}
\title{Operator Learning for Backstepping Control of Parabolic Equations on Time-Varying Domains }
\author{Jinrun Yan,\quad Kai Liu,\quad Zhong-Jie Han \\
  School of Mathematics and KL-AAGDM, Tianjin University\\
  Tianjin 300350, P. R. China }
\date{}
\maketitle
\begingroup
\renewcommand{\thefootnote}{} 
  \begin{NoHyper}
    \footnotetext{E-mail: jryan@tju.edu.cn (J. Yan); k\_liu@tju.edu.cn (K. Liu); zjhan@tju.edu.cn (Z.J. Han).}
  \end{NoHyper}
\endgroup

\begin{abstract}
This paper develops a neural operator approach to boundary stabilization of a parabolic equation on a time-varying domain, reducing the cost of backstepping kernel computation. Specifically, we establish the continuity of the operator mapping moving boundaries to backstepping kernels and apply the universal approximation theorem to establish the existence of an approximate kernel. Due to the compactness requirement, the approximate kernel is derived only on a finite time interval, limiting the exponential decay estimate to a finite horizon. To extend the result to an infinite horizon, we must circumvent the compactness requirement. Taking periodic moving boundary as an example, we propose a modified DeepONet architecture that lifts the non-compact time domain into a compact, higher-dimensional space. This modification relaxes the constraint of approximation accuracy and extends finite-horizon exponential decay to uniform exponential stability over an infinite horizon. Numerical simulations demonstrate accelerated backstepping kernel computation and closed-loop stabilization.

\end{abstract}

\noindent {\bf Key words:}\hspace{2mm} parabolic equation; time-varying domain; backstepping
method; DeepONet; operator learning.
\section{Introduction}
\subsection{Background}
The control of parabolic partial differential equations (PDEs) is a hot topic of modern control theory \cite{coron2007control,krstic2008boundary}. Among various boundary control techniques, the PDE backstepping method has emerged as a highly systematic and powerful framework for designing explicit boundary feedback laws \cite{smyshlyaev2004closed,smyshlyaev2005backstepping}. While the stabilization of PDEs on fixed spatial domains has been extensively investigated, considerably fewer results are available for PDEs defined on time-varying domains\cite{koga2018control,meurer2009tracking}.
Such moving-boundary problems naturally arise in a wide range of engineering applications, including Stefan problems for phase changes \cite{koga2020materials}, crystal growth processes \cite{maidi2010boundary}, catalyst shrinkage \cite{izadi2015pde}, and chemical species generation and consumption \cite{armaou2001computation}. Periodic moving boundaries arise in habitats subject to periodic environmental changes \cite{allwright2025periodic} and soft porous materials under cyclic loading \cite{fiori2023periodic}. These examples motivate the study of  backstepping-based control for parabolic PDEs on time-periodic domains.

However, the practical implementation of backstepping-based controllers on time-varying domains faces a severe computational bottleneck. For systems with fixed domains and time-invariant parameters, the backstepping kernel $k(x,y)$ is time-invariant and can be computed offline \cite{smyshlyaev2004closed}. In contrast, a moving boundary induces a time-varying backstepping transformation \cite{espitia2019boundary}. In practice, evaluating this kernel equation necessitates dynamic grid remeshing and continuous numerical integration at each time instant \cite{jadachowski2012efficient,woittennek2017approximation}. The repeated online solution of this second-order PDE introduces a prohibitive computational burden, rendering real-time feedback control largely infeasible.

To alleviate the computational burden of traditional PDE solvers, Neural Operators (NOs) offer a data-driven approach to learning parameter-dependent backstepping kernels\cite{lu2021learning}.
Notably, the Deep Operator Network (DeepONet) architecture \cite{deng2022approximation} has demonstrated exceptional efficiency in surrogate modeling for complex PDEs \cite{lu2021learning}. Recent pioneering work by Krstic and his collaborators has demonstrated the potential of neural operators for approximating backstepping kernels, eliminating the need for online numerical solutions of kernel PDEs in first-order hyperbolic \cite{wang2025backstepping} and parabolic systems \cite{bhan2023neural,bhan2025adaptive,krstic2024neural}. Their contributions have further extended operator-learning frameworks to systems with delays \cite{wang2025deep} and adaptive control schemes \cite{bhan2023operator}. Lyu et al. extended NO-based backstepping to reaction–diffusion \cite{lyu2026operator} and Markov jump hyperbolic PDE–ODE systems \cite{lyu2026robust}.

Despite these remarkable advancements, existing NO-based control frameworks are restricted to fixed spatial domains. To the best of our knowledge, approximating the kernel operator on a time-varying domain remains an open and profound challenge. 

\subsection{Contributions and Novelties}
The main contributions are summarized as follows:

(1) We establish the existence of a neural operator with the time-varying boundary as input and the corresponding kernel as output. We further prove that the output of this neural operator can approximate the exact backstepping kernel with arbitrary accuracy. Based on the approximated kernel, we further construct a boundary feedback controller and rigorously prove that the resulting closed-loop system admits an exponential decay estimate on arbitrary prescribed time interval.

(2) For periodic moving boundaries, we prove that the exact backstepping kernel inherits the same temporal periodicity and establish uniform exponential stability through periodic extension of the approximate feedback controller. However, the DeepONet universal approximation theorem does not guarantee the periodicity of the approximate kernel, which imposes a stricter constraint on the approximation accuracy. To address this issue, a novel DeepONet architecture is proposed, extending the finite-time exponential decay estimate to uniform exponential stability over an
infinite horizon.

(3) Numerical simulations are conducted using DeepONet to demonstrate the computational efficiency of the proposed approach. Compared with conventional numerical methods \cite{izadi2015pde,jadachowski2012efficient}, the proposed approach achieves a substantial computational speedup in our examples. This significantly reduces the computational burden by eliminating the need to numerically online solve time-varying kernel PDEs.

\subsection{Organization}
The remainder of this paper is organized as follows. Section \ref{sec:problem} introduces the parabolic equation on a time-varying domain and the corresponding backstepping kernel equation. Section \ref{sec:approx_stab} presents the NO approximation of the kernel operator and the stability analysis of the closed-loop system under the NO-based controller, including the case of periodic moving boundaries. Section \ref{sec:simulation} provides numerical simulations. Finally, Section \ref{sec:conclusion} concludes the paper.

\section{Preliminaries}
\label{sec:problem}
In this section, we first introduce the parabolic equation defined on a time-varying domain and then give the backstepping-based stabilization results from \cite{izadi2015pde}.


Consider the following one dimensional parabolic system on the time-varying domain:
\begin{equation}
  \label{eq:u_system_original}
  \begin{cases}
    u_t(x,t)=\alpha u_{xx}(x,t)-\dot l(t)u_x(x,t)+\lambda_0 u(x,t),
    & 0<x<l(t),\ t>t_0,\\
    u(0,t)=0,\\[1mm]
    u_x(l(t),t)=\tilde{U}(t),
  \end{cases}
\end{equation}
where $u(x,t)$ denotes the system state, which can be used to represent the concentration in diffusion–reaction processes \cite{armaou2001computation}, the temperature in crystal growth \cite{izadi2015pde}, or the temperature in phase-change processes \cite{koga2020materials}. The initial value $u(x,t_0)=u_0(x)$ with $t_0\geq0$, $\alpha>0$ is the diffusion coefficient, $\lambda_0>0$ is the reaction coefficient and  $\tilde{U}(t)$ is the boundary control input. $l(t)$ denotes stictly positive moving boundaries function, and $\dot{l}(t)$ denotes its time derivative. Additionally, we make the following assumption for moving boundaries function $l(t)$, which is made in \cite[Assumption 2]{izadi2015pde}.
\begin{assumption}
  \label{assum:stability}
  The moving boundaries function $l(t)$ is analytic and satisfies the following condition: there exists a positive constant $\bar D$, independent of $t$ and $j$, such that for all $t>t_0$ and all non-negative integers $j$,
  \[
  |\partial_t^j l(t)|\le \bar D^{j+1}j!.
  \]
\end{assumption}
To eliminate advection term in \eqref{eq:u_system_original}, {we first introduce the following transformation}
\begin{equation}
  \label{eq:advection transformation}
  u(x, t) = v(x, t)e^{\int_0^x \frac{\dot{l}(t)}{2\alpha} dy} = v(x, t)e^{\frac{\dot{l}(t)x}{2\alpha}},
\end{equation}
which further yields that
\begin{equation}
  \label{eq:plant_v1}
  \begin{cases}
    v_t(x,t)=\alpha v_{xx}(x,t)+\lambda(x,t)v(x,t), & 0<x<l(t),\ t>t_0,\\[1mm]
    v(0,t)=0,\\[1mm]
    v_x(l(t),t)+\dfrac{\dot l(t)}{2\alpha}v(l(t),t)=U(t),
  \end{cases}
\end{equation}
with the initial value $v(x,t_0)=v_0(x)$ and $\lambda(x, t) = \lambda_0 - \frac{\dot{l}^2(t) + 2\ddot{l}(t)x}{4\alpha}$. Since the term $e^{\frac{\dot l(t)x}{2\alpha}}$ is bounded in $L^\infty(0,l(t))$ for all \(t>t_0\), it follows from \eqref{eq:advection transformation} that the exponential stability of \(v(x,t)\) implies the exponential stability of \(u(x,t)\). Following \cite{izadi2015pde}, consider the backstepping transformation for system \eqref{eq:plant_v1}:
\begin{equation}
  \label{eq:exact_transform}
  w(x,t)=v(x,t)-\int_0^x k(x,y,t)v(y,t)\,dy
\end{equation}
to derive the following target system:
\begin{equation}
  \label{eq:target_exact}
  \begin{cases}
    w_t(x,t)=\alpha w_{xx}(x,t)-cw(x,t), & 0<x<l(t), \ t>t_0,  \\[1mm]
    w(0,t)=0,\\[1mm]
    w_x(l(t),t)=-\dfrac{\dot l(t)}{2\alpha}w(l(t),t)
  \end{cases}
\end{equation}
with the initial value $w(x,t_0)=w_0(x)$. The exponential stability of the target system in the $L^2$-norm sense was proven in \cite[Lemma 4]{izadi2015pde}. For initial data $w(\cdot,0)\in L^2(0,l(t_0))$, the well-posedness of the system \eqref{eq:target_exact} is guaranteed, see \cite{ng2012optimal}. It is worth noting that the backstepping transformation \eqref{eq:exact_transform} converts original system \eqref{eq:plant_v1} into the stable target system \eqref{eq:target_exact} with a tuning parameter $c$, which is directly related to the decay rate of system \eqref{eq:target_exact}.
The kernel function $k(x, y, t)$ satisfies
\begin{equation}
  \label{eq:kernel_pde_original}
  \begin{cases}
    k_t(x,y,t)=\alpha\big(k_{xx}(x,y,t)-k_{yy}(x,y,t)\big)-(\lambda(y,t)+c)k(x,y,t),\\[1mm]
    k(x,0,t)=0,\\[1mm]
    k(x,x,t)= -\frac{1}{2\alpha}\int_0^x (\lambda(s,t)+c)\,ds.
  \end{cases}
\end{equation}
Here, the kernel $k(x,y,t)$ is defined on the time-varying triangular domain $S(t):=\{(x,y):0\le y\le x\le l(t)\}$. Similar to \cite{meurer2009tracking}, the initial conditions at $t=t_0$ for $v(x,t), w(x,t)$ and $k(x,y,t)$ have to satisfy the constraint
\begin{equation}
  \int_{0}^{x} k(x,y, t_0)v_0(y)\mathrm{d}y = v_0(x) - w_0(x), \quad x \in [0,l(t_0)].
\end{equation}
According to \cite{izadi2015pde}, the state-feedback control can be correspondingly defined as follows:
\begin{equation}
  \label{eq:control_original}
  U(t)=
  \int_0^{l(t)}
  \left[
  \frac{\dot l(t)}{2\alpha}k(l(t),y,t)+k_x(l(t),y,t)
  \right]v(y,t)\,dy
  +k(l(t),l(t),t)v(l(t),t).
\end{equation}
Then, inspired by \cite{izadi2015pde}, introduce the following transformation
\begin{equation}
  \label{eq:transformation}
  \xi = \frac{x+y}{l(t)}, \quad \eta = \frac{x-y}{l(t)},
\end{equation}
where $(\xi,\eta)\in\bar S:=\{(\xi,\eta):0\le \eta\le \xi\le 2-\eta\}$. {Under the transformation \eqref{eq:transformation}, the original kernel \(k(x,y,t)\) on the time-varying domain \(S(t)\) is mapped to a kernel on the fixed reference domain \(\bar S\). Specifically, we define
  \begin{equation}\label{kdefinition}
    K(\xi, \eta, t)
    := k\left(\frac{l(t)}{2}(\xi+\eta),\frac{l(t)}{2}(\xi-\eta),t\right)
  \end{equation}
  for \((\xi,\eta)\in\bar S\). In the sequel, \(K(\xi, \eta, t)\) is used to denote the transformed kernel in the \((\xi,\eta)\) coordinates.} By \cite{izadi2015pde}, we get that kernel function $K(\xi, \eta, t)$ satisfies

\begin{equation}\label{eq:xietat}
  \begin{cases}
    \partial_t K(\xi, \eta, t) = \frac{4\alpha}{l^2(t)}\partial_\xi\partial_\eta K(\xi, \eta, t) + \frac{\dot{l}(t)}{l(t)}\big(\xi\partial_\xi K(\xi, \eta, t) + \eta\partial_\eta K(\xi, \eta, t)\big)- \bar{\lambda}(\xi, \eta, t)K(\xi, \eta, t),\\
    K(\xi,\xi, t) = 0, \\
    K(\xi, 0, t) = f(\xi, t),
  \end{cases}
\end{equation}
where
\begin{align}\label{lambdaf}
  \bar{\lambda}\big(\xi, \eta, t\big)
  &= -\frac{\dot{l}^2(t) + l(t)\ddot{l}(t)(\xi - \eta)}{4\alpha} + \lambda_0 + c, \enspace
  f\big(\xi, t\big) = \frac{\xi l(t)}{4\alpha} \left[ \frac{2\dot{l}^2(t) + \xi l(t)\ddot{l}(t)}{8\alpha} - (\lambda_0 + c) \right].
\end{align}
Meanwhile, we also obtain from \cite[Theorem 5]{izadi2015pde} that \eqref{eq:xietat} admits a solution $K(\xi, \eta, t)\in C^2\big(\bar{S}\times[t_0,T]\big)$ constructed by the successive integration method. Furthermore, {there exists a constant $M>0$ such that the solution to \eqref{eq:xietat} satisfies the uniform bound}
\begin{equation}
  \label{bound}
  \|K(\cdot,\cdot,t)\|_{\infty}+\|\partial_t K(\cdot,\cdot,t)\|_{\infty} \le M
  \qquad \forall t\in[t_0,T],
\end{equation}
where $\|\cdot\|_{\infty}$ denotes the supremum norm. Thanks to the analyticity of $l(t)$ and the invertibility of the mapping \eqref{eq:transformation}, the well-posedness established in the $(\xi, \eta)$ coordinates is mathematically equivalent to that in the original $(x, y)$ coordinates.


\section{Stabilization via NO-based controller}
\label{sec:approx_stab}

In this section, we construct an approximate kernel for the exact kernel $K(\xi,\eta,t)$. Based on this approximate kernel, we show that the corresponding boundary feedback control law still ensures the exponential stability of system \eqref{eq:plant_v1}.

\subsection{Approximation of backstepping kernel} \label{subsec:deeponet}

Let $\mathcal{U}$ be a compact subset of the real analytic function space $C^\omega([t_0,T])$ consisting of strictly positive functions, equipped with the metric induced by the $C^2([t_0,T])$ norm for any given $T>t_0$. Define the kernel operator
\begin{equation}\label{k}
  \begin{split}
    &\mathcal{K}:\mathcal{U}\to \mathcal{X}_K, \\
    &\mathcal{K}(l):= K_l,
  \end{split}
\end{equation}
where $K_l$ denotes the solution to \eqref{eq:xietat} constructed by the successive integration method with respect to the function $l(t)$ and
\begin{equation*}
  \mathcal{X}_K = \big\{K \in C(\bar S\times[t_0,T]) \mid K(\xi,\xi, t) = 0, \; \forall(\xi,\eta)\in\bar S, \forall t \in [t_0, T] \big\}.
\end{equation*}
Since the kernel $K_l$ is constructed from $l(t)$ by the successive integration method, the operator $\mathcal{K}$ is well-defined. As is common in operator learning \cite{lu2021learning}, the value of the output function $\mathcal K(l)$ at $(\xi,\eta,t)$ is written as $\mathcal K(l)(\xi,\eta,t)=K_l(\xi,\eta,t)$. The same convention is used for the operators $\mathcal M$, $\hat{\mathcal K}$ and $\hat{\mathcal M}$ below. Based on the above definition, we further give the following lemma to explain the continuity of the operator $\mathcal{K}$.

\begin{lemma}\label{lem:continue}
  Consider the kernel operator \(\mathcal K\) defined in \eqref{k}. Then, for any given $T>t_0$, the operator satisfies $ \mathcal K\in C\bigl(\mathcal{U},\,C(\bar S\times[t_0,T])\bigr)$.
\end{lemma}
\begin{proof}
  By \cite[Theorem 5]{izadi2015pde}, for $l\in\mathcal U$,
  \begin{equation}\label{series}
    K_{l}(\xi, \eta, t) = \sum_{n=1}^{\infty} K_{l,n}(\xi, \eta, t),
  \end{equation}
  where $K_{l,1}(\xi, \eta, t) := f(\xi, t) - f(\eta, t)$ and
  \begin{equation}\label{g}
    \begin{aligned}
      K_{l,n+1}(\xi, \eta, t)
      &:= \int_{\eta}^{\xi} \int_{0}^{\eta} \frac{l^2(t)}{4\alpha} \partial_t K_{l,n}(\rho, \sigma, t) +
      \left( \frac{\dot{l}(t) l(t)}{2\alpha} + \frac{l^2(t) \bar{\lambda}(\rho, \sigma, t)}{4\alpha} \right) K_{l,n}(\rho, \sigma, t)  \,d\sigma d\rho \\
      &\quad - \frac{\dot{l}(t) l(t)}{4\alpha} \left[ \xi \int_{0}^{\eta} K_{l,n}(\xi, \sigma, t) \, d\sigma - \eta \int_{0}^{\eta} K_{l,n}(\eta, \sigma, t) \, d\sigma + \eta \int_{\eta}^{\xi} K_{l,n}(\rho, \eta, t) \, d\rho \right].
    \end{aligned}
  \end{equation}
  By the proof of \cite[Theorem 5]{izadi2015pde}, the series \eqref{series} converges uniformly on $\bar S\times[t_0,T]$ and uniformly in $l\in\mathcal U$, that is,
  \begin{equation*}
    \lim_{N\to\infty}\sup_{l\in\mathcal U}\bigg\|K_l-\sum_{n=1}^{N}K_{l,n}\bigg\|_{\infty}=0.
  \end{equation*}
  Since a uniform limit of continuous maps is continuous, it suffices to show that, for every $n\ge1$, the map $l\mapsto K_{l,n}$ is continuous from $(\mathcal U,\|\cdot\|_{C^2([t_0,T])})$ to $(C(\bar S\times[t_0,T]),\|\cdot\|_{\infty})$. By \eqref{lambdaf}, \eqref{g} and induction on $n$, $K_{l,n}(\xi,\eta,t)$ is a multivariate polynomial in $\xi,\eta,l(t),\dots,l^{(n+1)}(t)$. Since $\xi,\eta\in[0,2]$ and Assumption \ref{assum:stability} holds uniformly over $\mathcal U$, there exists $L_n>0$ such that
  \begin{equation}\label{lip-n}
    \|K_{l,n}-K_{\tilde l,n}\|_{\infty}\le L_n\sum_{k=0}^{n+1}\|(l-\tilde l)^{(k)}\|_{C([t_0,T])},\quad l,\tilde l\in\mathcal U.
  \end{equation}

  Let $l,\tilde l\in\mathcal U$ and $k\ge1$. By the Gagliardo--Nirenberg inequality \cite{nirenberg1959elliptic} on $[t_0,T]$, there exists $C>0$ such that
  \begin{equation*}
    \|(l-\tilde l)^{(k)}\|_{C([t_0,T])}\le C\Big(\|l-\tilde l\|_{C([t_0,T])}^{1/2}\|(l-\tilde l)^{(2k)}\|_{C([t_0,T])}^{1/2}+\|l-\tilde l\|_{C([t_0,T])}\Big).
  \end{equation*}
  By Assumption \ref{assum:stability}, $\|(l-\tilde l)^{(2k)}\|_{C([t_0,T])}\le 2\bar D^{2k+1}(2k)!$. Hence, there exists $C_k>0$ such that
  \begin{equation}\label{gn}
    \|(l-\tilde l)^{(k)}\|_{C([t_0,T])}\le C_k\Big(\|l-\tilde l\|_{C([t_0,T])}^{1/2}+\|l-\tilde l\|_{C([t_0,T])}\Big).
  \end{equation}
  Combining \eqref{lip-n} and \eqref{gn}, we obtain
  \begin{equation*}
    \|K_{l,n}-K_{\tilde l,n}\|_{\infty}\le L_n\Big(1+\sum_{k=1}^{n+1}C_k\Big)\Big(\|l-\tilde l\|_{C([t_0,T])}^{1/2}+\|l-\tilde l\|_{C([t_0,T])}\Big).
  \end{equation*}
  Since $\|l-\tilde l\|_{C([t_0,T])}\le\|l-\tilde l\|_{C^2([t_0,T])}$, the map $l\mapsto K_{l,n}$ is continuous with respect to $\|\cdot\|_{C^2([t_0,T])}$. This completes the proof.
\end{proof}

Define the  operator $\mathcal M:\mathcal{U}\to C(\bar S\times[t_0,T]) \times C([0,2]\times[t_0,T]) \times C(\bar S\times[t_0,T])$ by
\begin{equation}\label{m}
  \mathcal M(l)
  :=
  \big(
  \mathcal{K}(l),\ \kappa_{0}(l),\ \kappa_{1}(l)
  \big),
\end{equation}
where
\begin{align*}
  \kappa_{0}(l)(\xi,t)
  &:=
  \frac{4\alpha}{l(t)}\partial_\xi \bigg( K_l(\xi,0,t)
  -
  f(\xi,t)\bigg),\\
  \kappa_{1}(l)(\xi,\eta,t)
  &:=
  \bigg(\partial_t
  -
  \frac{4\alpha}{l^2(t)}
  \partial_{\xi\eta}
  -
  \frac{\dot l(t)}{l(t)}
  \left(
  \xi\partial_\xi
  +
  \eta\partial_\eta
  \right)
  +
  \bar\lambda(\xi,\eta,t)\bigg)K_l(\xi,\eta,t).
\end{align*}
In particular, $\kappa_{0}(l)$ and $\kappa_{1}(l)$ are directly derived from the kernel equation \eqref{eq:xietat} to ensure the continuity of the operator $\mathcal M$.

Here, we list the following theorem to explain that there exists a neural operator that can approximate the operator $\mathcal M$.
\begin{theorem} \label{thm:deeponet_uat}
  For any $\varepsilon > 0$, there exists a neural operator $\hat{\mathcal{M}}=\big( \hat{\mathcal{K}}, \hat{\kappa}_0, \hat{\kappa}_1 \big)$ such that, for all $l\in\mathcal U$ and $(\xi, \eta, t) \in \bar S\times[t_0,T]$,
  \begin{equation} \label{eq:operator_bound}
    \big| \mathcal{M}(l)(\xi, \eta, t) - \hat{\mathcal{M}}(l)(\xi, \eta, t) \big| < \varepsilon,
  \end{equation}
  where $\hat K_l:=\hat{\mathcal{K}}(l)$, $\hat{\kappa}_0(l)(\xi,t):= \frac{4\alpha}{l(t)} \partial_\xi \big(\hat K_l(\xi,0,t) -f(\xi,t)\big)$, and $\hat{\kappa}_1(l)(\xi,\eta,t):= \big( \partial_t - \frac{4\alpha}{l^2(t)} \partial_{\xi\eta} - \frac{\dot l(t)}{l(t)} \big( \xi\partial_\xi + \eta\partial_\eta \big) + \bar\lambda(\xi,\eta,t)\big)\hat K_l(\xi,\eta,t)$. Equivalently, for all $l\in\mathcal U$ and $(\xi, \eta, t) \in \bar S\times[t_0,T]$, there exists a neural operator $\hat K_l$ such that $\hat K_l(\xi,\xi,t)=0,$ and
  \begin{equation} \label{eq:expanded_bound}
    \begin{aligned}
      &\big|\big(K_l - \hat K_l\big)(\xi, \eta, t) \big| + {\big|\frac{4\alpha}{l(t)}\partial_\xi\big(K_l-\hat K_l\big)(\xi,0,t)\big|} \\
      +  &\bigg| \Big( \partial_t - \frac{4\alpha}{l^2(t)}\partial_\xi\partial_\eta - \frac{\dot{l}(t)}{l(t)}(\xi\partial_\xi + \eta\partial_\eta)+\bar{\lambda}(\xi,\eta,t) \Big) \Big( K_l - \hat K_l \Big)(\xi, \eta, t) \bigg| < \varepsilon.
    \end{aligned}
  \end{equation}
\end{theorem}
\begin{proof}
  According to Lemma \ref{lem:continue}, together with \eqref{eq:xietat} and \eqref{m}, we can prove the continuity of the operator $\mathcal{M}$. Direct application of Universal Approximation Theorem in \cite{lu2021learning} leads to the proof.
\end{proof}





\subsection{Decay estimate over a finite horizon}
\label{subsec:stability}
Throughout this subsection, $l\in\mathcal U$ is fixed. For simplicity, we write $K:=K_l$ and $\hat K:=\hat K_l$, and denote by $k$ and $\hat k$ the corresponding kernels in the $(x,y)$ coordinates. Define
\begin{equation}
  \label{eq:approx_transform}
  \hat w(x,t)=v(x,t)-\int_0^x \hat k(x,y,t)v(y,t)\,dy,
\end{equation}
where
\begin{equation}\label{hatkdefinition}
  \hat k\left(
  \frac{l(t)}{2}(\xi+\eta),
  \frac{l(t)}{2}(\xi-\eta),
  t
  \right):=\hat K(\xi,\eta,t).
\end{equation}
The corresponding inverse backstepping transformation is defined as follows
\begin{equation}
  v(x,t) = \hat{w}(x,t) + \int_0^x \hat{q}(x,y,t)\hat{w}(y,t) \, dy. \label{eq:inverse_trans}
\end{equation}
Similar to \eqref{eq:control_original}, we define the approximate boundary feedback law by
\begin{equation}
  \label{eq:approx_control_original}
  \hat{U}(t)=
  \int_0^{l(t)}
  \left[
  \frac{\dot l(t)}{2\alpha}\hat k(l(t),y,t)+\hat k_x(l(t),y,t)
  \right]v(y,t)\,dy
  +\hat k(l(t),l(t),t)v(l(t),t).
\end{equation}
Under the control law \eqref{eq:approx_control_original}, the target system \eqref{eq:target_exact} can be written as
\begin{equation}
  \begin{cases}
    \label{eq:perturbed_target}
    \hat w_t(x,t)=\alpha \hat w_{xx}(x,t)-c\hat w(x,t)+\delta_0(x,t)v(x,t)
    - \int_0^x \delta_1(x,y,t)v(y,t)\,dy,\\[1mm]
    \hat w(0,t)=0, \\[1mm]
    \hat w_x(l(t),t)= -\frac{\dot l(t)}{2\alpha}\hat w(l(t),t),
  \end{cases}
\end{equation}
where
\begin{align*}
  \delta_0(x,t):&=
  \lambda(x,t)+c+2\alpha \frac{d}{dx} \hat k(x,x,t),\\
  \delta_1(x,y,t) :&= \hat k_t(x,y,t) - \alpha\big(\hat k_{xx}(x,y,t) - \hat k_{yy}(x,y,t)\big) + (\lambda(y,t)+c)\hat k(x,y,t).
\end{align*}
{ Then, using \eqref{eq:kernel_pde_original}, we obtain
  \begin{align}
    \delta_0(x,t)
    &=-2\alpha \frac{d}{dx} (k-\hat{k})(x,x,t),
    \label{eq:delta0_def}\\
    \delta_1(x,y,t)
    &= -\Big[
    \partial_t
    -\alpha(\partial_x^2-\partial_y^2)
    +\lambda(y,t)+c
    \Big](k-\hat{k})(x,y,t).
    \label{eq:delta1_def}
  \end{align}}
Before proving the stability of the original system \eqref{eq:plant_v1} under the control \eqref{eq:approx_control_original}, we first establish the following lemma to illustrate the boundedness of the kernel $\hat{k}$ and $\hat{q}$.

\begin{lemma}\label{lem:kernel_bounds}
  Consider the kernel $\hat{k}(x,y,t)$ defined in \eqref{hatkdefinition}  and $\hat{q}(x,y,t)$ defined in \eqref{eq:inverse_trans} over a time-varying spatial domain $S(t)$. Let Assumption \ref{assum:stability} hold. Then,  there exist positive constants $\tau_K$ and $\tau_Q$ such that
  \begin{align}
    \sup_{t\in[t_0,T]} \|\hat{k}(\cdot,\cdot,t)\|_{L^\infty(S(t))} &\le \tau_K, \label{eq:k_bound} \\
    \sup_{t\in[t_0,T]} \|\hat{q}(\cdot,\cdot,t)\|_{L^\infty(S(t))} &\le \tau_Q. \label{eq:l_bound}
  \end{align}
\end{lemma}

\begin{proof}
  Using \eqref{bound} and Theorem \ref{thm:deeponet_uat}, it follows that
  \begin{equation*}
    \|\hat{k}(\cdot,\cdot,t)\|_{L^\infty(S(t))} \le \|k(\cdot,\cdot,t)\|_{L^\infty(S(t))} + \|\hat{k}(\cdot,\cdot,t) - k(\cdot,\cdot,t)\|_{L^\infty(S(t))} < M + \varepsilon.
  \end{equation*}
  Defining $\tau_K := M + \varepsilon$, we directly obtain \eqref{eq:k_bound}. Following \cite[Section 4.5]{krstic2008boundary}, we derive that
  \begin{equation}
    \hat{q}(x,y,t) = \hat{k}(x,y,t) + \int_y^x \hat{q}(x,\sigma,t) \hat{k}(\sigma,y,t) \, d\sigma, \quad \forall\;(x,y,t)\in S(t)\times[t_0,T].
    \label{eq:resolvent_eq}
  \end{equation}
  Then, we obtain:
  \begin{equation}\label{est-3}
    |\hat{q}(x,y,t)| \le |\hat{k}(x,y,t)| + \int_y^x |\hat{q}(x,\sigma,t)| |\hat{k}(\sigma,y,t)| \, d\sigma.
  \end{equation}
  Substituting \eqref{eq:k_bound} into \eqref{est-3} yields:
  \begin{equation}\label{est-4}
    |\hat{q}(x,y,t)| \le \tau_K + \tau_K \int_y^x |\hat{q}(x,\sigma,t)| \, d\sigma.
  \end{equation}
  Applying Gronwall's inequality to \eqref{est-4}, we obtain the following explicit bound for the inverse kernel:
  \begin{equation*}
    |\hat{q}(x,y,t)| \le \tau_K \exp\left( \int_y^x \tau_K \, d\sigma \right) = \tau_K e^{\tau_K (x-y)}.
  \end{equation*}
  Since $x,y\in[0,l(t)]$ and the moving boundaries satisfies $\sup_{t\ge t_0} l(t)\le l^*$, it follows that $x-y\le l(t)\le l^*$. Therefore, we obtain \eqref{eq:l_bound} with $\tau_Q=\tau_K e^{\tau_K l^*}$. The proof is complete.
\end{proof}
Then, we present the main theorem to obtain the exponential stability of the system \eqref{eq:plant_v1} under the control \eqref{eq:approx_control_original}, which is listed as follows.
\begin{theorem}
  \label{thm:approx_stability} Consider the closed-loop system composed of plant \eqref{eq:plant_v1} and  control \eqref{eq:approx_control_original}. There exists a sufficiently small $\varepsilon^*>0$ such that, for any $t_0\ge0$ and any $T>t_0$, if the approximate kernel $\hat{k}$ satisfies \eqref{eq:expanded_bound} on $\bar S\times[t_0,T]$ with $\varepsilon\in(0,\varepsilon^*)$,
  then for any $v(\cdot,t_0)\in L^2(0,l(t_0))$, the closed-loop system satisfies
  \begin{equation*}
    \|v(\cdot, t)\|_{L^2(0, l(t))} \le C_I C_D e^{-\frac{\mu_0-C_0\varepsilon}{2} (t-t_0)} \|v(\cdot, t_0)\|_{L^2(0, l(t_0))}, \quad \forall t \in [t_0,T],
  \end{equation*}
  where $ C_I:= 1+\frac{\tau_Q l^*}{\sqrt2}$, $C_D:= 1+\frac{\tau_K l^*}{\sqrt2} $, $\mu_0 := \frac{4\alpha}{(l^*)^2} +2 c$, and $C_0:=2C_I(1+l^*)$. Here, $\varepsilon^*$, $C_I$, $C_D$, $\mu_0$ and $C_0$ depend only on $\alpha$, $c$, $l^*$ and $M$ in \eqref{bound}.
\end{theorem}

\begin{proof}
  From \eqref{eq:inverse_trans}, we have
  \begin{equation}\label{trans}
    v=\hat w+\hat{\mathcal Q}_t \hat w, \quad
    \mbox{where} \quad
    (\hat{\mathcal Q}_t \hat w)(x):=\int_0^x \hat q(x,y,t)\hat w(y,t)\,dy.
  \end{equation}
  For each $x\in[0,l(t)]$, utilizing \eqref{eq:l_bound} and applying the H\"{o}lder's inequality yields
  \begin{equation}\label{est-5}
    \left| \int_0^x \hat q(x,y,t)\hat w(y,t)\,dy \right|
    \le \tau_Q \int_0^x |\hat w(y,t)|\,dy
    \le \tau_Q\sqrt{x}\,\|\hat w(\cdot,t)\|_{L^2(0,l(t))}.
  \end{equation}
  Squaring \eqref{est-5} and integrating with respect to $x$ over $(0,l(t))$ gives
  \begin{equation*}
    \|\hat{\mathcal Q}_t\hat w\|_{L^2(0,l(t))}^2
    \le \tau_Q^2 \|\hat w(\cdot,t)\|_{L^2(0,l(t))}^2 \int_0^{l(t)} x\,dx
    = \frac{\tau_Q^2 l^2(t)}{2}\|\hat w(\cdot,t)\|_{L^2(0,l(t))}^2.
  \end{equation*}
  Taking the square root and bounding $l(t)$ by $l^*$, we obtain $\|\hat{\mathcal Q}_t\hat w\|_{L^2(0,l(t))} \le \frac{\tau_Q l^*}{\sqrt2}\|\hat w\|_{L^2(0,l(t))}$. Then, we have
  \begin{equation}
    \label{eq:inverse_norm_est}
    \begin{split}
      \|v(\cdot,t)\|_{L^2(0,l(t))} &\le \|\hat w\|_{L^2(0,l(t))} + \|\hat{\mathcal Q}_t\hat w\|_{L^2(0,l(t))} \\
      &\le C_I \|\hat w(\cdot,t)\|_{L^2(0,l(t))},
    \end{split}
  \end{equation}
  where $C_I := 1+\frac{\tau_Q l^*}{\sqrt2}$. Similarly, rewriting \eqref{eq:approx_transform} as $\hat w=v-\hat{\mathcal T}_t v$, with
  \begin{equation*}
    (\hat{\mathcal T}_t v)(x)=\int_0^x \hat k(x,y,t)v(y,t)\,dy.
  \end{equation*}
  We have $\|\hat w\|_{L^2(0,l(t))} \le \|v\|_{L^2(0,l(t))} + \|\hat{\mathcal T}_t v\|_{L^2(0,l(t))}$. Using \eqref{eq:k_bound} and Cauchy--Schwarz inequality yields
  \begin{equation*}
    \left| \int_0^x \hat k(x,y,t)v(y,t)\,dy \right|
    \le \tau_K\sqrt{x}\,\|v(\cdot,t)\|_{L^2(0,l(t))}.
  \end{equation*}
  Squaring and integrating over $(0,l(t))$ yields $\|\hat{\mathcal T}_t v\|_{L^2(0,l(t))} \le \frac{\tau_K l^*}{\sqrt2}\|v\|_{L^2(0,l(t))}$. Thus, we obtain
  \begin{equation}
    \label{eq:direct_norm_est}
    \|\hat w(\cdot,t)\|_{L^2(0,l(t))} \le C_D \|v(\cdot,t)\|_{L^2(0,l(t))}, \quad \text{where} \quad C_D := 1+\frac{\tau_K l^*}{\sqrt2}.
  \end{equation}

  Define the Lyapunov functional
  \begin{equation}\label{lyapunov}
    V(t):=\frac12\int_0^{l(t)} \hat w^2(x,t)\,dx.
  \end{equation}
  Differentiating \eqref{lyapunov} along \eqref{eq:perturbed_target} with respect to $t$ yields
  \begin{equation}
    \label{eq:Vdot_main}
    \dot V = -\alpha\int_0^{l(t)} \hat w_x^2(x,t)\,dx - c\int_0^{l(t)} \hat w^2(x,t)\,dx + \Delta_0(t) + \Delta_1(t),
  \end{equation}
  where
  \begin{equation}
    \label{delta0}
    \Delta_0(t):=\int_0^{l(t)} \hat w(x,t)\delta_0(x,t)v(x,t)\,dx,
  \end{equation}
  \begin{equation}
    \label{delta1}
    \Delta_1(t):=\int_0^{l(t)} \hat w(x,t) \left(\int_0^x -\delta_1(x,y,t)v(y,t)\,dy\right)dx.
  \end{equation}
  Subsequently, we need to estimate $\delta_0$ and $\delta_1$ defined in \eqref{eq:delta0_def} and \eqref{eq:delta1_def}, respectively. For the second term on the left of inequality \eqref{eq:expanded_bound}, by \eqref{kdefinition} and \eqref{hatkdefinition}, it yields
  \begin{equation}\label{est-10}
    \frac{4\alpha}{l(t)}
    \partial_\xi (K-\hat{K})(\xi,0,t)
    =
    2\alpha \frac{d}{dx} (k-\hat{k})(x,x,t).
  \end{equation}
  Then,  we have

  \begin{equation}\label{est-8}
    \left(
    \partial_t
    -
    \frac{4\alpha}{l^2(t)}\partial_{\xi\eta}
    -
    \frac{\dot l(t)}{l(t)}
    \left(
    \xi\partial_\xi+\eta\partial_\eta
    \right)
    \right) (K-\hat{K})(\xi,\eta,t)
    =
    \big(\partial_t-\alpha(\partial_x^2-\partial_y^2)\big) (k-\hat{k})(x,y,t).
  \end{equation}
  Moreover, from the definition of $\bar\lambda$ in \eqref{lambdaf},
  \begin{equation}\label{est-9}
    \bar\lambda(\xi,\eta,t)
    =
    \lambda\left(\frac{l(t)}{2}(\xi-\eta),t\right)+c
    =
    \lambda(y,t)+c.
  \end{equation}
  Therefore, for the third term of the left of inequality \eqref{eq:expanded_bound}, by \eqref{est-8} and \eqref{est-9}, it gives
  \begin{align}\label{est-11}
    &\left(
    \partial_t
    -
    \frac{4\alpha}{l^2(t)}\partial_{\xi\eta}
    -
    \frac{\dot l(t)}{l(t)}
    \left(
    \xi\partial_\xi+\eta\partial_\eta
    \right)
    +
    \bar\lambda(\xi,\eta,t)
    \right)
    (K-\hat{K})(\xi,\eta,t) \\
    =&
    \left(
    \partial_t
    -
    \alpha(\partial_x^2-\partial_y^2)
    +
    \lambda(y,t)+c
    \right)
    (k-\hat{k})(x,y,t).
  \end{align}
  By \eqref{eq:expanded_bound}, \eqref{est-10} and \eqref{est-11}, we obtain
  \begin{equation}\label{equi}
    \big|
    2\alpha
    (k-\hat{k})_x(x,x,t)
    \big|+
    \bigg|
    \Big(
    \partial_t-\alpha(\partial_x^2-\partial_y^2)
    +\lambda(y,t)+c
    \Big)
    (k-\hat{k})(x,y,t)
    \bigg|
    <\varepsilon,
  \end{equation}
  which implies
  \begin{equation}\label{delta}
    |\delta_0(x,t)|<\varepsilon,
    \qquad
    |\delta_1(x,y,t)|<\varepsilon
  \end{equation}
  for \(0\le y\le x\le l(t)\) and \(t\in[t_0,T]\). By the H\"{o}lder's inequality, \eqref{eq:inverse_norm_est}, \eqref{delta} and \eqref{delta0}, it gives
  \begin{equation}\label{est-6}
    \begin{aligned}
      |\Delta_0(t)|
      &\le \|\delta_0(\cdot,t)\|_{L^\infty(0,l(t))} \|\hat w(\cdot,t)\|_{L^2(0,l(t))} \|v(\cdot,t)\|_{L^2(0,l(t))} \\
      &\le \varepsilon C_I \|\hat w(\cdot,t)\|_{L^2(0,l(t))}^2 = 2\varepsilon C_I V(t).
    \end{aligned}
  \end{equation}
  Moreover, similar to \eqref{est-6} and \eqref{delta1}, we get
  \begin{equation}\label{est-7}
    \begin{aligned}
      |\Delta_1(t)|
      &\le \varepsilon \sqrt{l^*} \|v(\cdot,t)\|_{L^2(0,l(t))} \int_0^{l(t)} |\hat w(x,t)|\,dx \\
      &\le \varepsilon l^* \|v(\cdot,t)\|_{L^2(0,l(t))} \|\hat w(\cdot,t)\|_{L^2(0,l(t))} \le 2\varepsilon l^* C_I V(t).
    \end{aligned}
  \end{equation}
  Combining \eqref{eq:Vdot_main} with \eqref{est-6} and \eqref{est-7} yields
  \begin{equation}\label{eq:V28}
    \dot V \le -\alpha\|\hat w_x(\cdot,t)\|_{L^2(0,l(t))}^2 - c\|\hat w(\cdot,t)\|_{L^2(0,l(t))}^2 + 2\varepsilon C_I(1+l^*)V(t).
  \end{equation}
  Thanks to $\hat w(0,t)=0$,  applying the Poincar\'e inequality, we have
  \begin{equation}\label{3.29}
    \|\hat w(\cdot,t)\|_{L^2(0,l(t))}^2 \le \frac{(l^*)^2}{2}\|\hat w_x(\cdot,t)\|_{L^2(0,l(t))}^2.
  \end{equation}
  Substituting \eqref{3.29} into \eqref{eq:V28}, it follows that
  \begin{equation*}
    \dot V \le -\frac{4\alpha}{(l^*)^2}V - 2cV + 2\varepsilon C_I(1+l^*)V = -(\mu_0-C_0\varepsilon)V,
  \end{equation*}
  where $\mu_0:=\frac{4\alpha}{(l^*)^2}+2c$ and $C_0:=2C_I(1+l^*)$. It can be calculated that $V(t)\le V(t_0)e^{-{(\mu_0-C_0\varepsilon)}{(t-t_0)}}$.  {It remains to verify that $\mu_0-C_0\varepsilon>0$. By the definitions of $C_0$, $C_I$, $\tau_Q$, and $\tau_K$, we obtain
    \begin{equation*}
      \mu_0-C_0\varepsilon
      =
      \frac{4\alpha}{(l^*)^2}+2c
      -
      2(1+l^*)\left(
      1+\frac{l^*(M+\varepsilon)e^{(M+\varepsilon)l^*}}{\sqrt{2}}
      \right)\varepsilon .
    \end{equation*}
    Therefore, there exists $\varepsilon^*>0$ such that, for every $\varepsilon\in(0,\varepsilon^*)$,
    \begin{equation}\label{eq:epsilon}
      \mu_0-C_0\varepsilon>0.
    \end{equation}
  } Finally, using \eqref{eq:inverse_norm_est}, we get
  \begin{equation*}
    \|v(\cdot, t)\|_{L^2(0,l(t))} \le C_I \|\hat{w}(\cdot, t)\|_{L^2(0,l(t))} = C_I \sqrt{2V(t)} \le C_I \sqrt{2V(t_0)} e^{-\frac{\mu_0-C_0\varepsilon}{2} (t-t_0)}.
  \end{equation*}
  From \eqref{eq:direct_norm_est}, we conclude that
  \begin{equation*}
    \|v(\cdot, t)\|_{L^2(0,l(t))} \le C_I C_D e^{-\frac{\mu_0-C_0\varepsilon}{2} (t-t_0)} \|v(\cdot, t_0)\|_{L^2(0,l(t_0))}.
  \end{equation*}
  The proof is complete.
\end{proof}
\subsection{Periodic extension and stabilization over an infinite horizon}
\label{subsec:periodicity}
In this subsection, we consider that $l(t)$ is periodic. By periodicity, we first show that the kernel $K_l(\xi,\eta,t)$ is also periodic with respect to $t$. 
Then, we redesign the structure of DeepONet to ensure the periodicity can be inherited by the approximate kernel. Finally, we design the new controller by periodic extension and establish the uniform exponential stability of the closed-loop system.

First, the following assumption is made.
\begin{assumption}\label{assum:periodic}
  There exists a constant $P>0$ such that $l(t+P)=l(t)$ for all $t\ge0$.
\end{assumption}

In what follows, we denote the set $\mathcal U$ whose elements are the restrictions to $[0,P]$ of moving boundaries satisfying Assumptions \ref{assum:stability} and \ref{assum:periodic}.
Then, in accordance with \eqref{hatkdefinition}, we define $\hat{k}(x,y,t)$ and subsequently obtain $\hat{U}(t)$ on $[0,P)$ by substituting $\hat{k}(x,y,t)$ for $k(x,y,t)$ in \eqref{eq:control_original}. Then, for any $t\in[mP,(m+1)P)$ with  $m\in\mathbb N_0$, we define
\begin{equation}\label{eq:periodic_control}
  \begin{aligned}
    \hat U_{per}(t)\!=\!\int_0^{l(t)}\Big[\frac{\dot l(t)}{2\alpha}\hat k(l(t),y,t-mP)+\hat k_x(l(t),y,t-mP)\Big]v(y,t)dy\!+\!\hat k(l(t),l(t),t-mP)v(l(t),t),\\ t\in[mP,(m+1)P),\ m\in\mathbb N_0.
  \end{aligned}
\end{equation}

Under \eqref{eq:periodic_control}, the closed-loop system reads
\begin{equation}\label{eq:closed_loop}
  \begin{cases}
    \begin{aligned}[t]v_t(x,t)=\alpha v_{xx}(x,t)+\lambda(x,t)v(x,t),\enspace
    0<x<l(t),\ t>0,\end{aligned}\\[1mm]
    v(0,t)=0,\\[1mm]
    \begin{aligned}[t]v_x(l(t),t)+\dfrac{\dot l(t)}{2\alpha}v(l(t),t)=\hat U_{per}(t), 
    t>0,\end{aligned}
  \end{cases}
\end{equation}
with $v(x,t_0)=v_0(x)$, $t_0\ge0$. 
\begin{definition}\label{def:ues}{\hspace{-0.05cm}\rm \cite{sanz2019observation}}
  The closed-loop system \eqref{eq:closed_loop} is said to be uniformly exponentially stable if there exist constants $C>0$ and $\mu>0$ such that, for any $t_0\ge0$ and $v_0\in L^2(0,l(t_0))$, the solution of \eqref{eq:closed_loop} satisfies
  \begin{equation}\label{eq:ues}
    \|v(\cdot,t)\|_{L^2(0,l(t))}\le Ce^{-\mu(t-t_0)}\|v_0\|_{L^2(0,l(t_0))}, \; \forall\, t\ge t_0.
  \end{equation}
\end{definition}

Based on the periodic assumption of $l(t)$, we further obtain the kernel function $K_l(\cdot,\cdot,t)$ is periodic in time $t$, which is described by the following lemma.

\begin{lemma}\label{lem:periodic}
  Let Assumptions \ref{assum:stability} and \ref{assum:periodic} hold. Then, the solution $K_l(\xi,\eta,t)$ to \eqref{eq:xietat} satisfies
  \begin{equation}\label{eq:periodic}
    K_l(\xi,\eta,t+P)=K_l(\xi,\eta,t), \qquad \forall\,(\xi,\eta)\in\bar S,\ t\ge0.
  \end{equation}
\end{lemma}
\begin{proof}
  Recall \eqref{series}, we first show that for all $n\ge1$,
  \begin{equation}\label{eq:periodic_n}
    K_{l,n}(\xi,\eta,t+P)=K_{l,n}(\xi,\eta,t), \qquad \forall\,(\xi,\eta)\in\bar S,\ t\ge0.
  \end{equation}

  Differentiating $l(t+P)=l(t)$ with respect to $t$ gives $\dot l(t+P)=\dot l(t)$ and $\ddot l(t+P)=\ddot l(t)$ for $t\ge0$. Then, by \eqref{lambdaf}, $f(\xi,t+P)=f(\xi,t)$ and $\bar\lambda(\xi,\eta,t+P)=\bar\lambda(\xi,\eta,t)$ for $t\ge0$, and thus \eqref{eq:periodic_n} holds for $n=1$. Assume that \eqref{eq:periodic_n} holds for some $n\ge1$; then $\partial_t K_{l,n}(\xi,\eta,t+P)=\partial_t K_{l,n}(\xi,\eta,t)$ for $t\ge0$. Since the integrals in \eqref{g} are taken only with respect to $(\rho,\sigma)$, $K_{l,n+1}(\xi,\eta,t)$ is determined by $K_{l,n}(\cdot,\cdot,t)$, $\partial_t K_{l,n}(\cdot,\cdot,t)$, $l(t)$, $\dot l(t)$ and $\bar\lambda(\cdot,\cdot,t)$ at the same time $t$. Therefore, \eqref{eq:periodic_n} holds for $n+1$. By induction, \eqref{eq:periodic_n} holds for all $n\ge1$. Since the series $\sum_{n=1}^{\infty}K_{l,n}$ converges uniformly on $\bar S\times[0,T]$ for every $T>0$ \cite[Theorem 5]{izadi2015pde}, we obtain from \eqref{eq:periodic_n} that
  \[
  \begin{aligned}
    K_l(\xi,\eta,t+P)
    =\lim_{N\to\infty}\sum_{n=1}^{N}K_{l,n}(\xi,\eta,t+P)
    =\lim_{N\to\infty}\sum_{n=1}^{N}K_{l,n}(\xi,\eta,t)= K_l(\xi,\eta,t),
  \end{aligned}
  \]
  which establishes \eqref{eq:periodic}. This completes the proof.
\end{proof}

Lemma \ref{lem:periodic} shows that the exact kernel $K_l(\cdot,\cdot,t)$ is periodic. However, when directly applying the DeepONet universal approximation theorem, the approximate kernel $\hat{K}(\cdot,\cdot,t)$ cannot inherit the periodicity. For example, 
\begin{equation}\label{discon}
\begin{aligned}
& |\hat K_l(\xi,\eta,P)-\hat K_l(\xi,\eta,0)| \\
    \le & |\hat K_l(\xi,\eta,P)-K_l(\xi,\eta,P)|
   + |K_l(\xi,\eta,0)-\hat K_l(\xi,\eta,0)|\\
    <& 2\varepsilon.
\end{aligned}
\end{equation}

To guarantee the uniformly exponential stability, a more restrictive condition on the approximation accuracy $\epsilon$ needs to be imposed, which is stated by the following theorem.

\begin{theorem}\label{thm:uniform}
  Let Assumptions \ref{assum:stability} and \ref{assum:periodic} hold, and let $\hat K_l$ satisfy \eqref{eq:expanded_bound} on $\bar S\times[0,P]$ with $\varepsilon>0$ sufficiently small such that
  \begin{equation}\label{eq:eps_condition}
    \mu_0-C_0\varepsilon-\frac{2}{P}\ln(1+\sqrt2\,\varepsilon l^*C_I)>0.
  \end{equation}
  Then, the closed-loop system \eqref{eq:closed_loop} is uniformly exponentially stable, and \eqref{eq:ues} holds with
  \begin{equation*}
    C=C_IC_D(1+\sqrt2\,\varepsilon l^*C_I), \qquad \mu=\frac12\Big(\mu_0-C_0\varepsilon-\frac{2}{P}\ln(1+\sqrt2\,\varepsilon l^*C_I)\Big),
  \end{equation*}
  where $C_I$, $C_D$, $\mu_0$ and $C_0$ are given in Theorem \ref{thm:approx_stability} with $[t_0,T]$ replaced by $[0,P]$.
\end{theorem}
\begin{proof}
  For $t\ge t_0$ with $t\in[mP,(m+1)P)$, $m\in\mathbb N_0$, we periodically extend the transformation \eqref{eq:approx_transform} by setting
  \begin{equation}\label{eq:w_per}
    \hat w(x,t):=v(x,t)-\int_0^x\hat k(x,y,t-mP)v(y,t)\,dy,\qquad V(t):=\frac12\int_0^{l(t)}\hat w^2(x,t)\,dx.
  \end{equation}
  By Assumption \ref{assum:periodic}, Lemma \ref{lem:kernel_bounds}, \eqref{eq:inverse_norm_est}, \eqref{eq:direct_norm_est}, and the proof of Theorem \ref{thm:approx_stability}, we obtain
  \begin{align}
    &C_I^{-1}\|v(\cdot,t)\|_{L^2(0,l(t))}\le\sqrt{2V(t)}\le C_D\|v(\cdot,t)\|_{L^2(0,l(t))},\qquad t\ge t_0,\label{eq:per_equiv}\\
    &V(t)\le e^{-(\mu_0-C_0\varepsilon)(t-s)}V(s),\qquad \max\{t_0,mP\}\le s\le t<(m+1)P.\label{eq:per_decay}
  \end{align}

  Let $mP>t_0$. Since $v(x,t)$ and $l(t)$ are continuous in $t$, and $\hat K_l(\xi,\eta,t)$ is continuous on $\bar S\times[0,P]$, it follows from \eqref{eq:w_per}, Assumption \ref{assum:periodic} and \eqref{hatkdefinition} that
  \begin{align*}
    \hat w(x,mP^-) &= v(x,mP)-\int_0^x\hat k(x,y,P)v(y,mP)\,dy,\\
    \hat w(x,mP) &= v(x,mP)-\int_0^x\hat k(x,y,0)v(y,mP)\,dy,
  \end{align*}
  for $x\in[0,l(mP)]$. Subtracting the first equality from the second yields
  \begin{equation}\label{eq:w_jump_diff}
    \hat w(x,mP)-\hat w(x,mP^-) = \int_0^x \big[\hat k(x,y,P)-\hat k(x,y,0)\big]v(y,mP)\,dy.
  \end{equation}
 By \eqref{kdefinition}, \eqref{hatkdefinition}, Lemma \ref{lem:periodic} and \eqref{eq:expanded_bound}, we obtain
  \begin{align*}
    |\hat k(x,y,P)-\hat k(x,y,0)|
    &= |\hat K_l(\xi,\eta,P)-\hat K_l(\xi,\eta,0)| \\
    &\le |\hat K_l(\xi,\eta,P)-K_l(\xi,\eta,P)|  + |K_l(\xi,\eta,0)-\hat K_l(\xi,\eta,0)| \\
    &< 2\varepsilon.
  \end{align*}
  Substituting this estimate into \eqref{eq:w_jump_diff} and applying the H\"older inequality, we have
  \begin{equation*}
    |\hat w(x,mP)-\hat w(x,mP^-)| \le 2\varepsilon \int_0^x |v(y,mP)|\,dy \le 2\varepsilon\sqrt{x}\,\|v(\cdot,mP)\|_{L^2(0,l(mP))}.
  \end{equation*}
  Squaring both sides, integrating with respect to $x$ over $(0,l(mP))$ yield
  \begin{equation*}
    \|\hat w(\cdot,mP)-\hat w(\cdot,mP^-)\|_{L^2(0,l(mP))} \le \sqrt{2}\,\varepsilon l^*\|v(\cdot,mP)\|_{L^2(0,l(mP))}.
  \end{equation*}
  Recalling the definition of $V(t)$ in \eqref{eq:w_per} and applying the triangle inequality in $L^2(0,l(mP))$, we arrive at
  \begin{equation}\label{eq:per_jump}
    \begin{aligned}
      \sqrt{2V(mP)} &= \|\hat w(\cdot,mP)\|_{L^2(0,l(mP))} \\
      &\le \|\hat w(\cdot,mP^-)\|_{L^2(0,l(mP))} + \|\hat w(\cdot,mP)-\hat w(\cdot,mP^-)\|_{L^2(0,l(mP))} \\
      &\le \sqrt{2V(mP^-)} + \sqrt{2}\,\varepsilon l^*\|v(\cdot,mP)\|_{L^2(0,l(mP))} \\
      &\le (1+\sqrt2\,\varepsilon l^*C_I)\sqrt{2V(mP^-)},
    \end{aligned}
  \end{equation}
  where the last inequality follows from letting $t\to mP^-$ in \eqref{eq:per_equiv} and the continuity of $v$ and $l$ at $t=mP$. Let $t\ge t_0$, and let $m_0,n\in\mathbb N_0$ be such that $t_0\in[m_0P,(m_0+1)P)$ and $t\in[nP,(n+1)P)$. By \eqref{eq:per_decay}, \eqref{eq:per_jump} and induction on $n$, we obtain
  \begin{equation*}
    V(t)\le(1+\sqrt2\,\varepsilon l^*C_I)^{2(n-m_0)}e^{-(\mu_0-C_0\varepsilon)(t-t_0)}V(t_0)\le(1+\sqrt2\,\varepsilon l^*C_I)^{2}e^{-2\mu(t-t_0)}V(t_0),
  \end{equation*}
  where the last inequality follows from $n-m_0<1+\frac{t-t_0}{P}$ and the definition of $\mu$. Combining this with \eqref{eq:per_equiv} and $v(\cdot,t_0)=v_0$, we obtain
  \begin{equation*}
    \|v(\cdot,t)\|_{L^2(0,l(t))}\le C_I\sqrt{2V(t)}\le C_I(1+\sqrt2\,\varepsilon l^*C_I)e^{-\mu(t-t_0)}\sqrt{2V(t_0)}\le Ce^{-\mu(t-t_0)}\|v_0\|_{L^2(0,l(t_0))},
  \end{equation*}
  where $\mu>0$ by \eqref{eq:eps_condition}. The proof is complete.
\end{proof}

\begin{remark}
By comparing \eqref{eq:epsilon} and \eqref{eq:eps_condition}, we notice that Theorem \ref{thm:uniform} imposes a stricter condition on the upper bound of the approximation accuracy $\epsilon$. A stricter upper bound on $\varepsilon$ demands a more accurate DeepONet, which in practice requires a higher training cost.   
\end{remark}

To relax the restriction on $\epsilon$, we encode time using sine and cosine features in the DeepONet trunk input. The resulting network architecture is illustrated in Fig.~\ref{fig:deeponet_architecture}.

\begin{figure}[htbp]
    \centering
    \includegraphics[width=\linewidth]{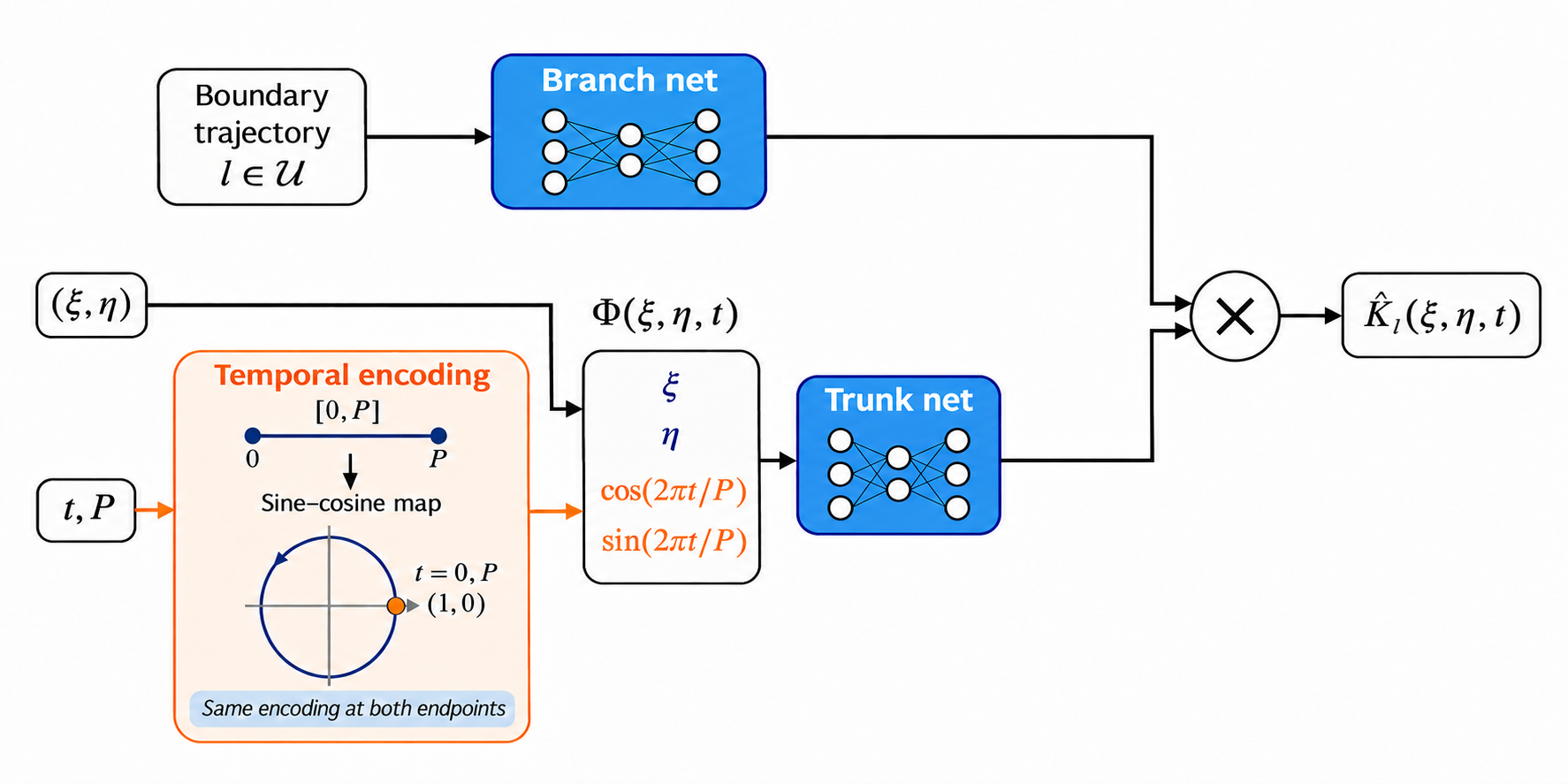}
    \caption{DeepONet architecture with sine--cosine time encoding.}
    \label{fig:deeponet_architecture}
\end{figure}



Specifically, we have the following main theorem:

\begin{theorem}\label{thm:periodic_uat}
  Let Assumptions \ref{assum:stability} and \ref{assum:periodic} hold. Then, for any $\varepsilon>0$, there exists a neural operator $\hat{\mathcal M}=\big(\hat{\mathcal K},\hat\kappa_0,\hat\kappa_1\big)$ such that, for every $l\in\mathcal U$, $\hat{\mathcal M}(l)$ is defined on $\bar S\times[0,P]$ and satisfies
  \begin{align}
    &\big|\mathcal M(l)(\xi,\eta,t)-\hat{\mathcal M}(l)(\xi,\eta,t)\big|<\varepsilon, \quad \forall\,(\xi,\eta,t)\in\bar S\times[0,P],\label{eq:periodic_operator_bound}\\
    &\hat{\mathcal M}(l)(\xi,\eta,0)=\hat{\mathcal M}(l)(\xi,\eta,P),\hspace{1.1cm} \forall\,(\xi,\eta)\in\bar S.\label{eq:hatM_endpoint}
  \end{align}
  In particular, $\hat K_l=\hat{\mathcal K}(l)$ satisfies \eqref{eq:expanded_bound} on $\bar S\times[0,P]$ and
  \begin{equation}\label{eq:hatK_periodic}
    \hat K_l(\xi,\eta,0)=\hat K_l(\xi,\eta,P),\qquad \forall\,(\xi,\eta)\in\bar S.
  \end{equation}
\end{theorem}
\begin{proof}
  Let $\mathbb S^1\subset\mathbb R^2$ denote the unit circle, and define the mapping $\Phi$ by
  \begin{equation*}
    \begin{aligned}
      \Phi : \bar S\times[0,P] &\to \bar S\times\mathbb S^1, \\
      (\xi,\eta,t) &\mapsto \Big(\xi,\eta,\cos\frac{2\pi t}{P},\sin\frac{2\pi t}{P}\Big).
    \end{aligned}
  \end{equation*}
  Clearly, $\Phi$ is continuous and surjective, and $\bar S\times\mathbb S^1\subset\mathbb R^4$ is compact. Since $t\mapsto\big(\cos\frac{2\pi t}{P},\sin\frac{2\pi t}{P}\big)$ is one-to-one on $[0,P)$, we have
  \begin{equation}\label{eq:Phi_identify}
    \Phi(\xi_1,\eta_1,t_1)=\Phi(\xi_2,\eta_2,t_2)\ \Longleftrightarrow\ (\xi_1,\eta_1)=(\xi_2,\eta_2)\ \text{and}\ \big(t_1=t_2\ \text{or}\ \{t_1,t_2\}=\{0,P\}\big).
  \end{equation}
For each $l\in\mathcal U_P$, there exists a unique mapping $\mathcal M_P(l):\bar S\times\mathbb S^1\to\mathbb R^3$ satisfying
  \begin{equation}\label{eq:MP_def}
    \mathcal M_P(l)\circ\Phi = \mathcal M(l),
  \end{equation}
i.e., $\mathcal M_P(l)\big(\Phi(\xi,\eta,t)\big)=\mathcal M(l)(\xi,\eta,t)$ for all $(\xi,\eta,t)\in\bar S\times[0,P]$.
  It can be directly verified that $\mathcal M_P(l)$ is well defined.
  We next show that $\mathcal M_P(l)\in C(\bar S\times\mathbb S^1)$. Let $\zeta_n\to\zeta$ in $\bar S\times\mathbb S^1$, and choose $p_n\in\bar S\times[0,P]$ with $\Phi(p_n)=\zeta_n$. Since $\bar S\times[0,P]$ is compact, every subsequence of $\{p_n\}$ has a further subsequence $\{p_{n_k}\}$ converging to some $p\in\bar S\times[0,P]$. By the continuity of $\Phi$, $\Phi(p)=\zeta$. Then, by the continuity of $\mathcal M(l)$,
  \begin{equation*}
    \mathcal M_P(l)(\zeta_{n_k})=\mathcal M(l)(p_{n_k})\to\mathcal M(l)(p)=\mathcal M_P(l)(\zeta),\qquad k\to\infty.
  \end{equation*}
  Therefore, $\mathcal M_P(l)(\zeta_n)\to\mathcal M_P(l)(\zeta)$, i.e., $\mathcal M_P(l)\in C(\bar S\times\mathbb S^1)$. Moreover, since $\Phi$ is onto, for $l_1,l_2\in\mathcal U_P$,
  \begin{equation*}
    \|\mathcal M_P(l_1)-\mathcal M_P(l_2)\|_{C(\bar S\times\mathbb S^1)}=\|\mathcal M(l_1)-\mathcal M(l_2)\|_{C(\bar S\times[0,P])}=\|K_{l_1}-K_{l_2}\|_{C(\bar S\times[0,P])},
  \end{equation*}
  and thus the operator $\mathcal M_P$ is continuous by Lemma \ref{lem:continue}. Since $\mathcal U$ and $\bar S\times\mathbb S^1$ are compact, $\mathcal M_P$ satisfies the conditions of \cite[Theorem 2.1]{deng2022approximation}. Hence, for any $\varepsilon>0$, there exists a neural operator $\hat{\mathcal M}_P$ such that
  \begin{equation}\label{eq:MP_bound}
    \big|\mathcal M_P(l)(\zeta)-\hat{\mathcal M}_P(l)(\zeta)\big|<\varepsilon,\qquad \forall\,l\in\mathcal U_P,\ \zeta\in \bar S\times\mathbb S^1.
  \end{equation}
  Let $\hat{\mathcal M}(l)(\xi,\eta,t):=\hat{\mathcal M}_P(l)\big(\Phi(\xi,\eta,t)\big)$ for $(\xi,\eta,t)\in\bar S\times[0,P]$. Then, \eqref{eq:periodic_operator_bound} follows from \eqref{eq:MP_def} and \eqref{eq:MP_bound}, and \eqref{eq:hatM_endpoint} follows from $\Phi(\xi,\eta,0)=\Phi(\xi,\eta,P)$. Taking the first component of \eqref{eq:hatM_endpoint} yields \eqref{eq:hatK_periodic}. The proof is complete.
\end{proof}
Based on Theorem \ref{thm:periodic_uat}, we obtain the following improved result.

\begin{theorem}\label{thm:uniform_periodic}
  Let Assumptions~\ref{assum:stability} and
  \ref{assum:periodic} hold, and let $\hat K_l$ satisfy
  \eqref{eq:expanded_bound} and \eqref{eq:hatK_periodic}
  on $\bar S\times[0,P]$ with $\varepsilon>0$
  sufficiently small such that
  \begin{equation}\label{newepsilon}
    \mu_0-C_0\varepsilon>0.
  \end{equation}
  Then, the closed-loop system \eqref{eq:closed_loop}
  under the periodically extended feedback
  \eqref{eq:periodic_control} is uniformly exponentially
  stable.
\end{theorem}

\begin{proof}
  For $t\in[mP,(m+1)P)$,
  $m\in\mathbb N_0$, define 
  \[
    \begin{aligned}
    &\hat w(x,t):=v(x,t)-\int_0^x\hat k(x,y,t-mP)v(y,t)\,dy,\\
    &\qquad V(t):=\frac12\int_0^{l(t)}\hat w^2(x,t)\,dx.
    \end{aligned}
  \]
  Within each period, similar to Theorem~\ref{thm:approx_stability}, it yields
  \begin{equation}\label{eq:periodic_matching_decay}
    \begin{aligned}
    &V(t)\le e^{-(\mu_0-C_0\varepsilon)(t-s)}V(s),
    \end{aligned}
  \end{equation}
where $\max\{t_0,mP\}\le s\le t<(m+1)P.$ It further implies that, for any $t_0\geq0$,
  \[
    V(t)\le
    e^{-(\mu_0-C_0\varepsilon)(t-t_0)}V(t_0).
  \]
 Moreover, the proof of Theorem~\ref{thm:approx_stability} gives
\begin{equation}\label{eq:periodic_matching_equiv}
    \begin{aligned}
C_I^{-1}\|v(\cdot,t)\|_{L^2(0,l(t))}\le\sqrt{2V(t)}\le C_D\|v(\cdot,t)\|_{L^2(0,l(t))}.
    \end{aligned}
  \end{equation}
  Finally, using \eqref{eq:periodic_matching_equiv},
  we obtain
  \[
  \begin{aligned}
\|v(\cdot,t)\|_{L^2(0,l(t))}\le&C_I\sqrt{2V(t)}\\
    &\le C_I
    e^{-\frac{\mu_0-C_0\varepsilon}{2}(t-t_0)}
    \sqrt{2V(t_0)}\\
    &\le C_IC_D
    e^{-\frac{\mu_0-C_0\varepsilon}{2}(t-t_0)}
    \|v_0\|_{L^2(0,l(t_0))}.
  \end{aligned}
  \]
This completes the proof.
\end{proof}

\section{Numerical Simulation}
\label{sec:simulation}

In this section, we use numerical simulations to validate the effectiveness of the proposed approach. Specifically, we first compare the computational time required to obtain the kernels using the conventional numerical method and the DeepONet-based approach. Then, we compare the closed-loop stabilization performance achieved by controllers constructed from the corresponding kernels.

The system parameters are chosen as $\alpha=1$, $\lambda_0=20$ and $c=10$, and the moving boundaries is defined by $l(t)=1+0.5\sin(\gamma\pi t)$, with period $P=2/\gamma$. The initial condition for the system \eqref{eq:plant_v1}  is
\[
v(x,0)=\sin\!\left(\frac{\pi x}{2l(0)}\right)
+0.2\sin\!\left(\frac{3\pi x}{2l(0)}\right),
\qquad 0\le x\le l(0)=1.
\]

To generate the dataset, we sample 1280 values of $\gamma$ over $(0.05,0.95)$. For each trajectory, the kernel is firstly computed over one period by the method presented in \cite{izadi2015pde}. 
Meanwhile, for each trajectory, we sample 4096 space–time points $(\xi,\eta,t)$ in $\bar S\times[0,P]$. These 1280 trajectories are divided into three sets: 1024 for training, 128 for validation, and 128 for testing. The value $\gamma=1/3$ is excluded from all three sets.

Using this dataset, we train a DeepONet on a single NVIDIA GeForce RTX 4080 SUPER. The branch and trunk networks each have five hidden layers of width 256, with Tanh activations and LayerNorm applied to the last four hidden layers. The network has a rank of 128 and includes a trainable output scale, for a total of 598,017 parameters. The branch input is the scalar $\gamma$. For the trunk input, we normalize the spatial coordinates by mapping $\xi\in[0,2]$ and $\eta\in[0,1]$ to $\xi-1\in[-1,1]$ and $2\eta-1\in[-1,1]$, respectively. We encode time using sine and cosine features to enforce periodicity with period $P$. The resulting trunk input is $\bigl(\xi-1,2\eta-1,\cos(2\pi t/P),\sin(2\pi t/P)\bigr)$, so that all trunk input components lie in $[-1,1]$.


The model is trained by minimizing
\[
\mathcal L=\mathcal L_K+0.01\mathcal L_{\kappa_0}
+\mathcal L_{\kappa_1},
\]
where $\mathcal L_K$ measures the kernel-value prediction
error. The feedback law \eqref{eq:approx_control_original}
depends on both the approximate kernel $\hat k$ and its
spatial derivative $\hat k_x$. Hence, accurate kernel-value
predictions alone do not ensure accurate feedback gains.
To account for the derivative dependence of the feedback law,
we include two additional terms,
$\mathcal L_{\kappa_0}$ and $\mathcal L_{\kappa_1}$,
in the training loss. These terms penalize the residuals
associated with the last two components of the operator
$\mathcal M$ in Theorem~\ref{thm:deeponet_uat},
both of which involve derivatives of the predicted kernel. Specifically,
$\mathcal L_{\kappa_0}$ penalizes violations of the
differentiated diagonal condition, whereas
$\mathcal L_{\kappa_1}$ penalizes the kernel PDE residual.
These terms impose differential constraints on the
predicted kernel and correspond to the perturbations
$\delta_0$ and $\delta_1$ in the stability analysis,
thereby connecting the training objective to the
closed-loop stability conditions. The weight $0.01$ is chosen empirically based on preliminary
numerical experiments to balance the diagonal-condition
residual penalty against the other loss terms.

After 200 epochs using Adam with an initial learning
rate of $10^{-3}$, the checkpoint at epoch 188 is
selected by minimizing the validation loss.

For each test trajectory $\gamma_j$, the relative kernel error is computed as
\begin{equation*}
  \label{eq:relerr}
  E_j=\left(
  \frac{\sum_{i=1}^{N_q}
    \bigl|\widehat K_{\gamma_j}(\mathbf q_i)
    -K_{\gamma_j}(\mathbf q_i)\bigr|^2}
  {\sum_{i=1}^{N_q}|K_{\gamma_j}(\mathbf q_i)|^2}
  \right)^{1/2},
  \qquad \mathbf q_i=(\xi_i,\eta_i,t_i),
\end{equation*}
where $N_q=4096$ and $K_{\gamma_j}$ is the numerical reference kernel. Over the 128 independent test trajectories, the mean of $E_j$ is $0.06\%$.

With the trained model, we compute the backstepping kernels for $\gamma=1/3$, which is not included among the sampled trajectories. Table~\ref{tab:time_comparison} compares the time required to solve the backstepping kernels using the conventional numerical method and the DeepONet-based method. As the output grid is refined, more spatial points are used to represent the kernel surfaces, increasing the number of kernel evaluations. Here, $h=\Delta\xi=\Delta\eta$ denotes the output-grid spacing. The output-grid spacing $h$ is progressively reduced to increase the number of kernel evaluations and assess computational efficiency under increasing workloads.

As shown in Table~\ref{tab:time_comparison}, when the grid spacing $h$ varies from $0.020$ to $0.005$, the time required by the numerical method increases from $3.20$ to $27.50\mathrm{s}$, whereas that required by DeepONet increases from $0.08$ to $1.29\mathrm{s}$. DeepONet achieves a speedup of $20.96\times$--$38.04\times$ over the numerical method. These results demonstrate that the proposed approach substantially accelerates the computation of backstepping kernels.

\begin{table}[H]
  \centering
  \begin{tabular}{lccc}
    \toprule
    \makecell{Output-grid\\spacing $h$} &
    \makecell{Numerical kernel\\computational time (s)} &
    \makecell{NO kernel\\computational time (s)} &
    Speedup \\           
    \midrule
    0.020 & 3.2043  & 0.0842 & 38.04 \\
    0.010 & 6.7982  & 0.3243 & 20.96 \\
    0.005 & 27.4989 & 1.2862 & 21.38 \\
    \bottomrule
  \end{tabular}
  \caption{Kernel generation times for different output-grid spacings}
  \label{tab:time_comparison}
\end{table}

In addition to computational efficiency, we further assess how accurately the DeepONet-predicted kernels approximate the numerical kernels. Figure~\ref{fig:kernel_comparison} shows the numerical kernel, the DeepONet-based kernel, and their absolute error as functions of $(\xi,\eta)$ at $t=0,1.5,3,4.5,\mathrm{s}$. For $\gamma=1/3$, the relative kernel error over the full domain is $0.04\%$. On the control surface $x=l(t)$, the relative errors in the kernel value $K$ and the derivative $k_x$ are $0.03\%$ and $0.07\%$, respectively. These results show that the DeepONet-based kernel closely approximates the numerical kernel.


\begin{figure}[!htbp]
  \centering

    \newcommand{\KernelPanel}[1]{%
        \includegraphics[
            width=0.28\linewidth,
            height=0.145\textheight,
            keepaspectratio
        ]{1107figures/#1}%
    }

  \begin{tabular}{
      @{}c@{\hspace{0.01\linewidth}}
      c@{\hspace{0.01\linewidth}}c@{}
    }
    \KernelPanel{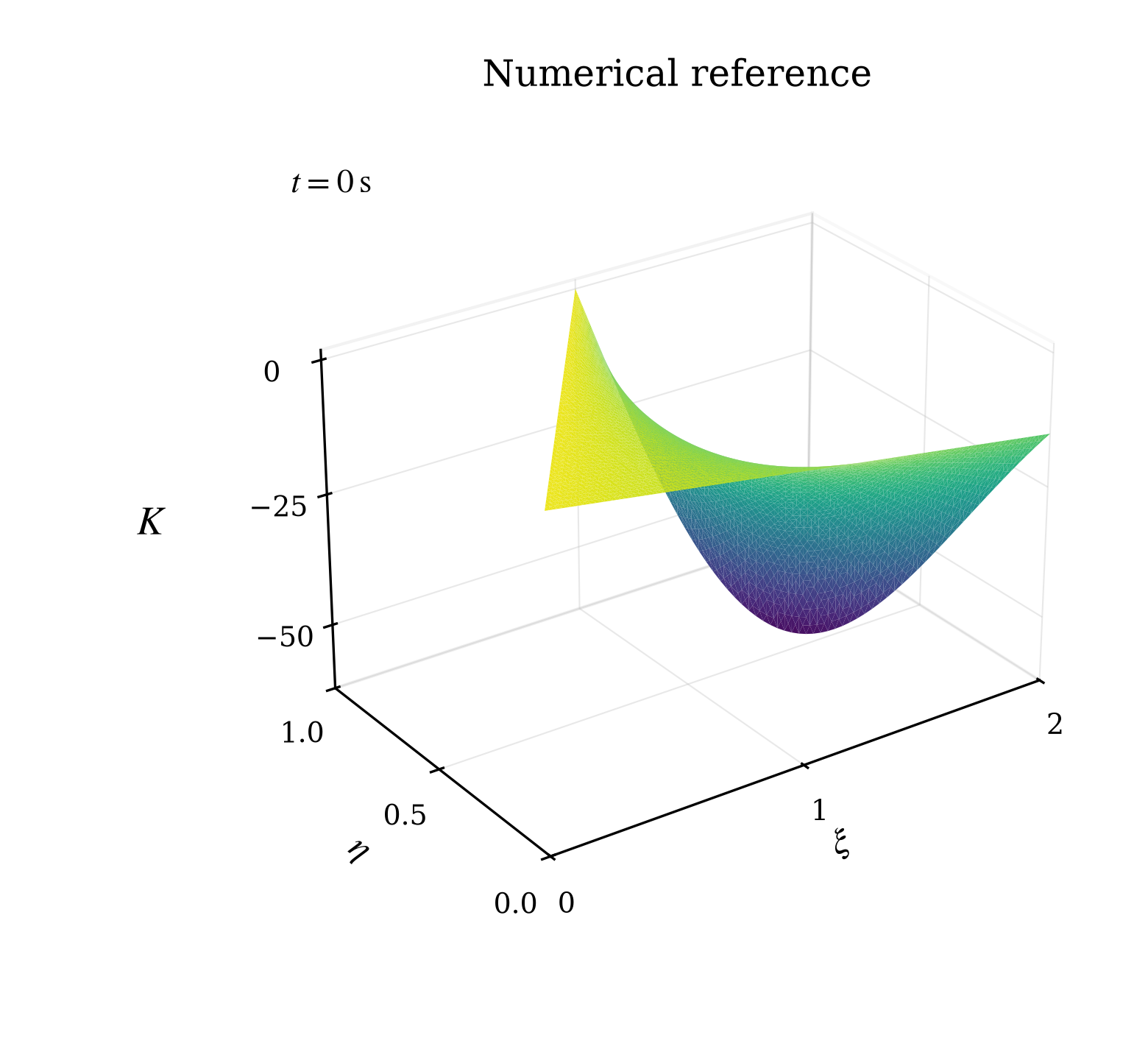} &
    \KernelPanel{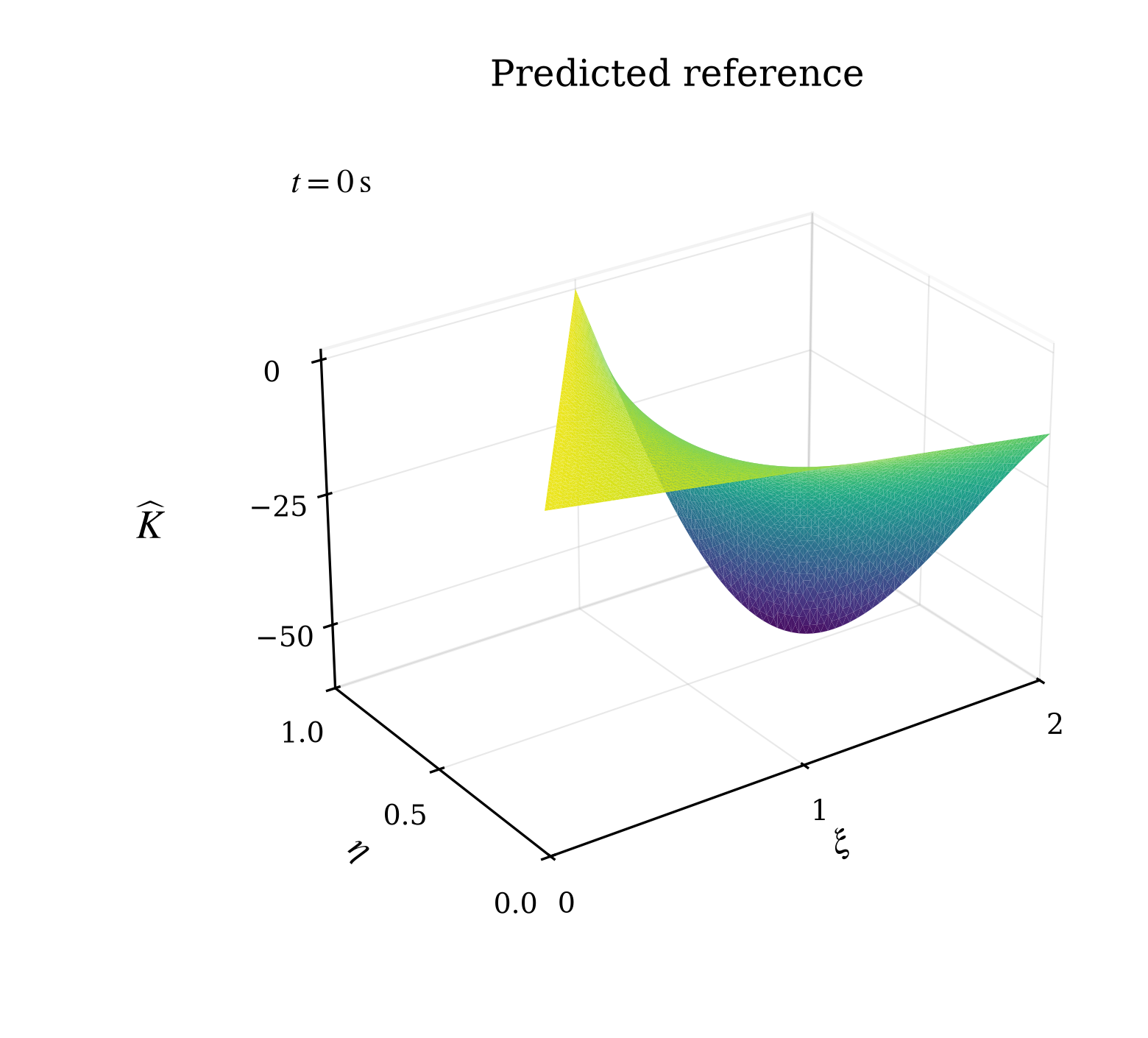} &
    \KernelPanel{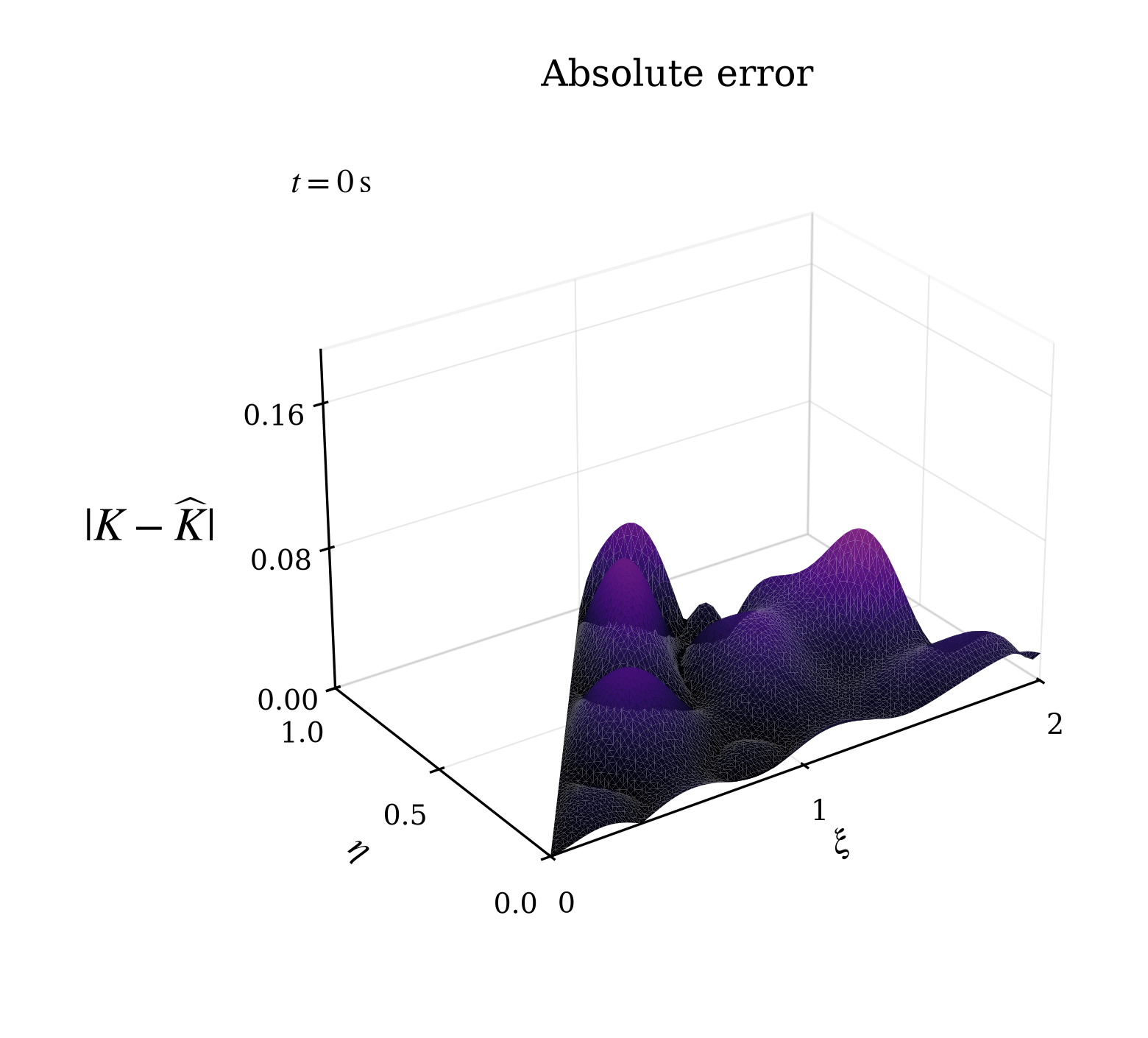}
    \\[0.5mm]

    \KernelPanel{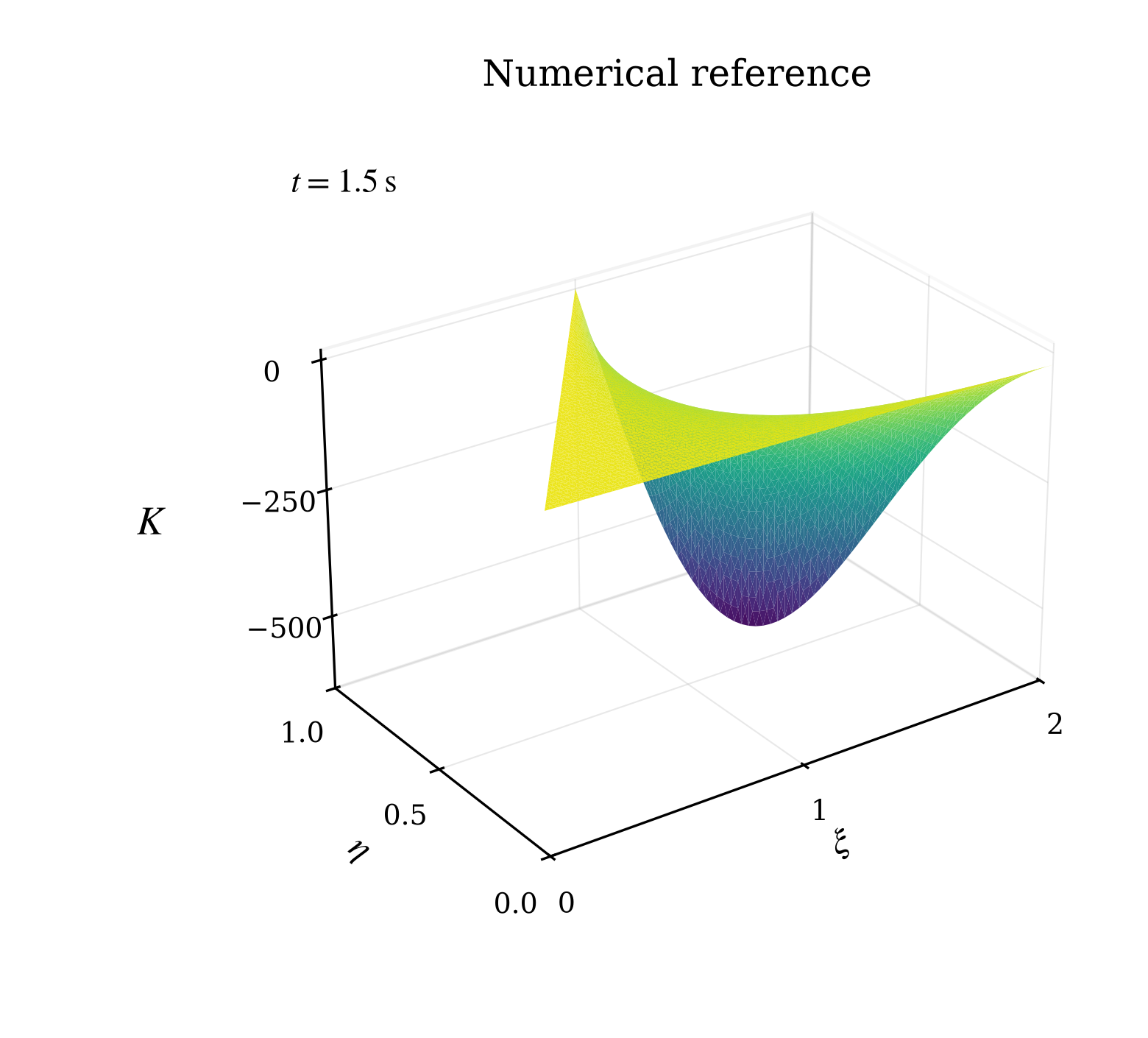} &
    \KernelPanel{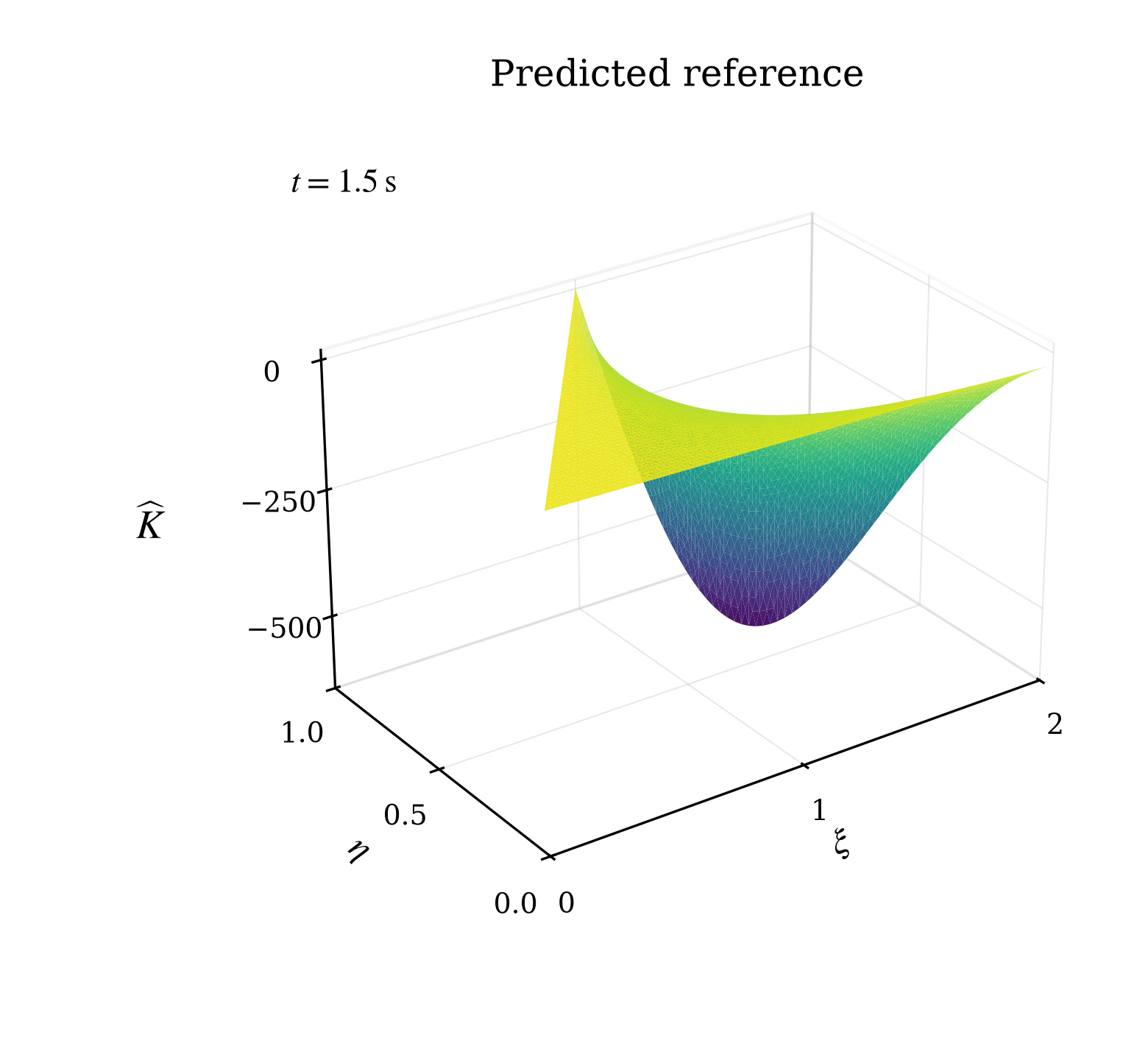} &
    \KernelPanel{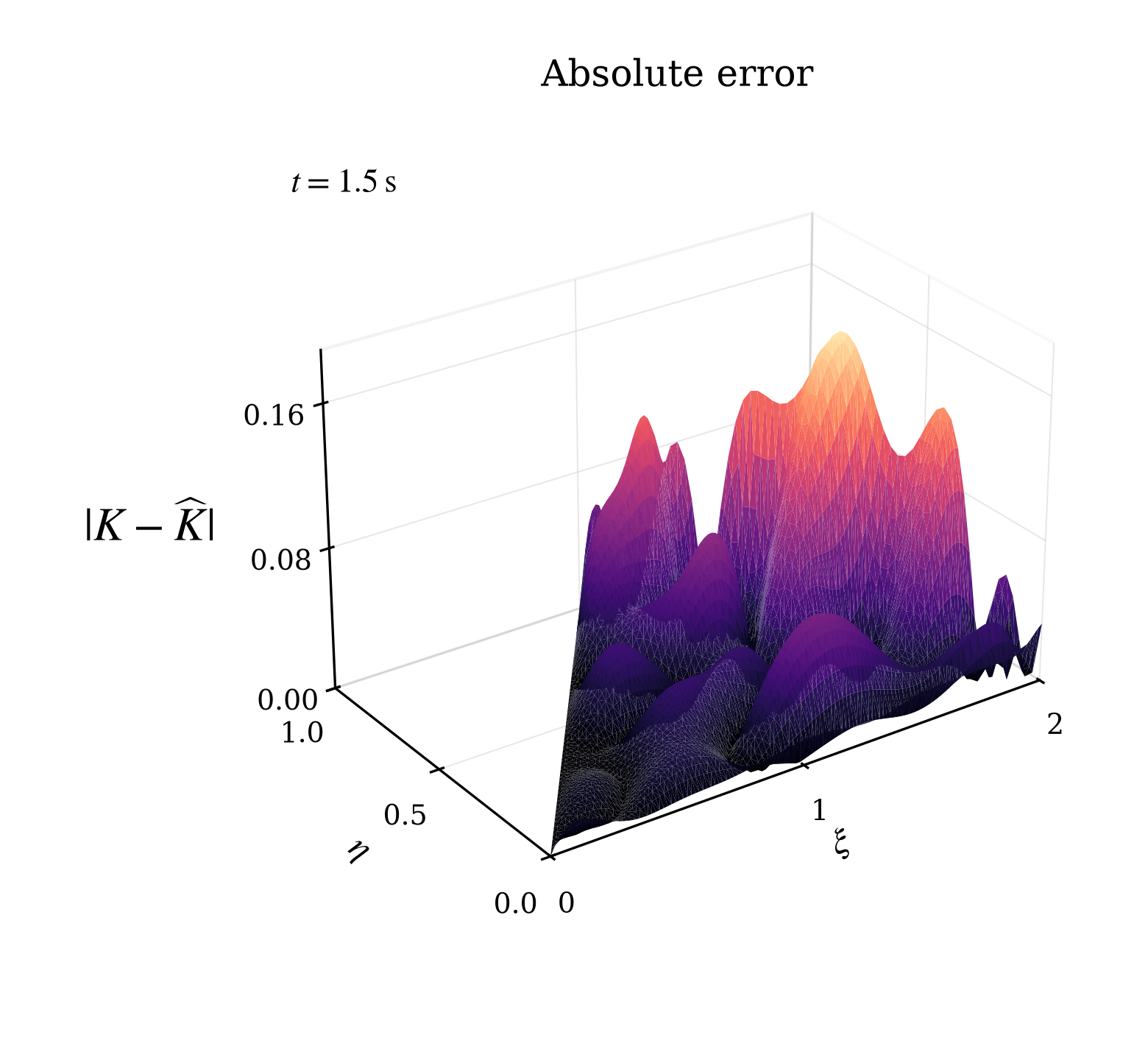}
    \\[0.5mm]

    \KernelPanel{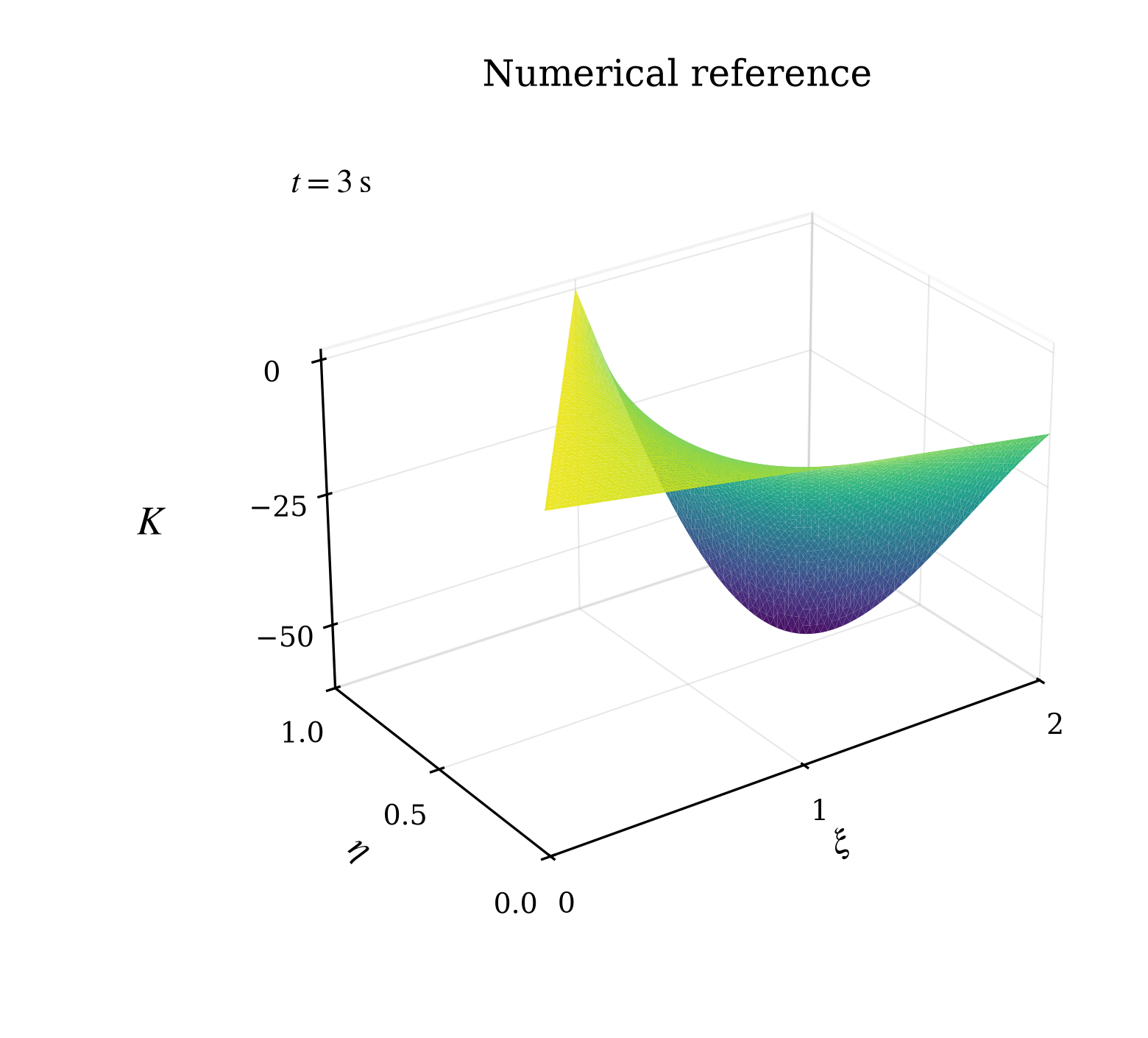} &
    \KernelPanel{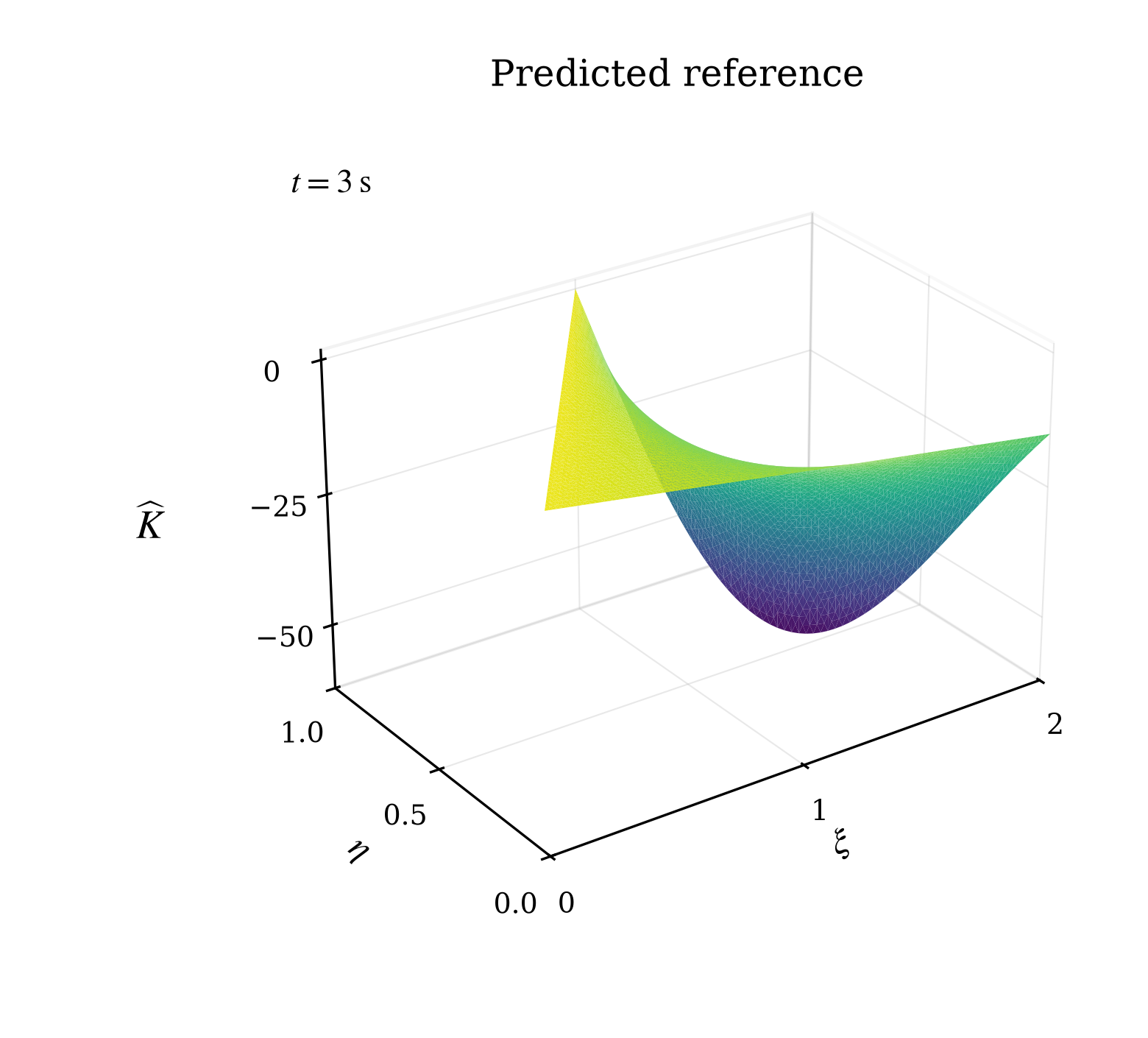} &
    \KernelPanel{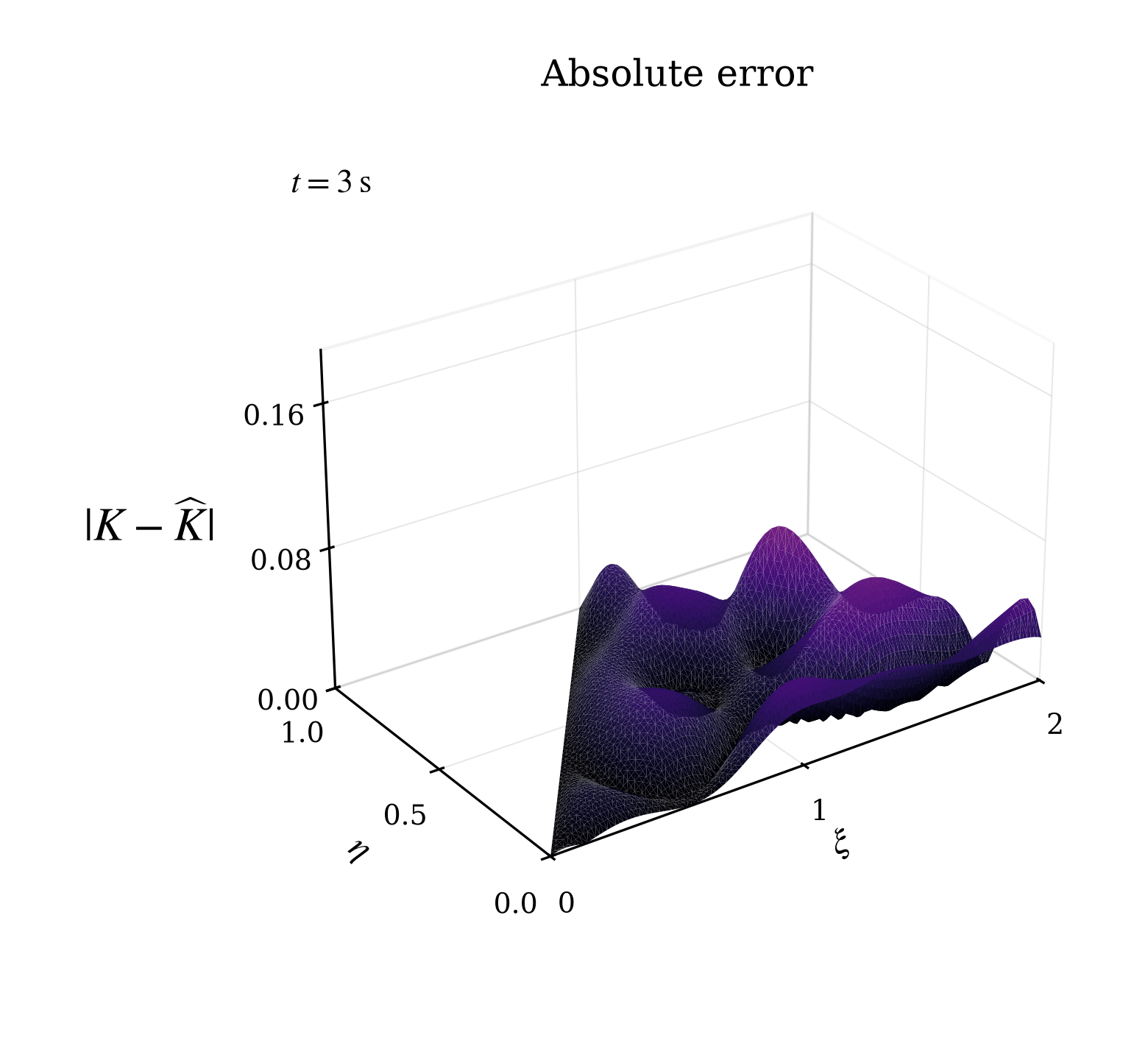}
    \\[0.5mm]

    \KernelPanel{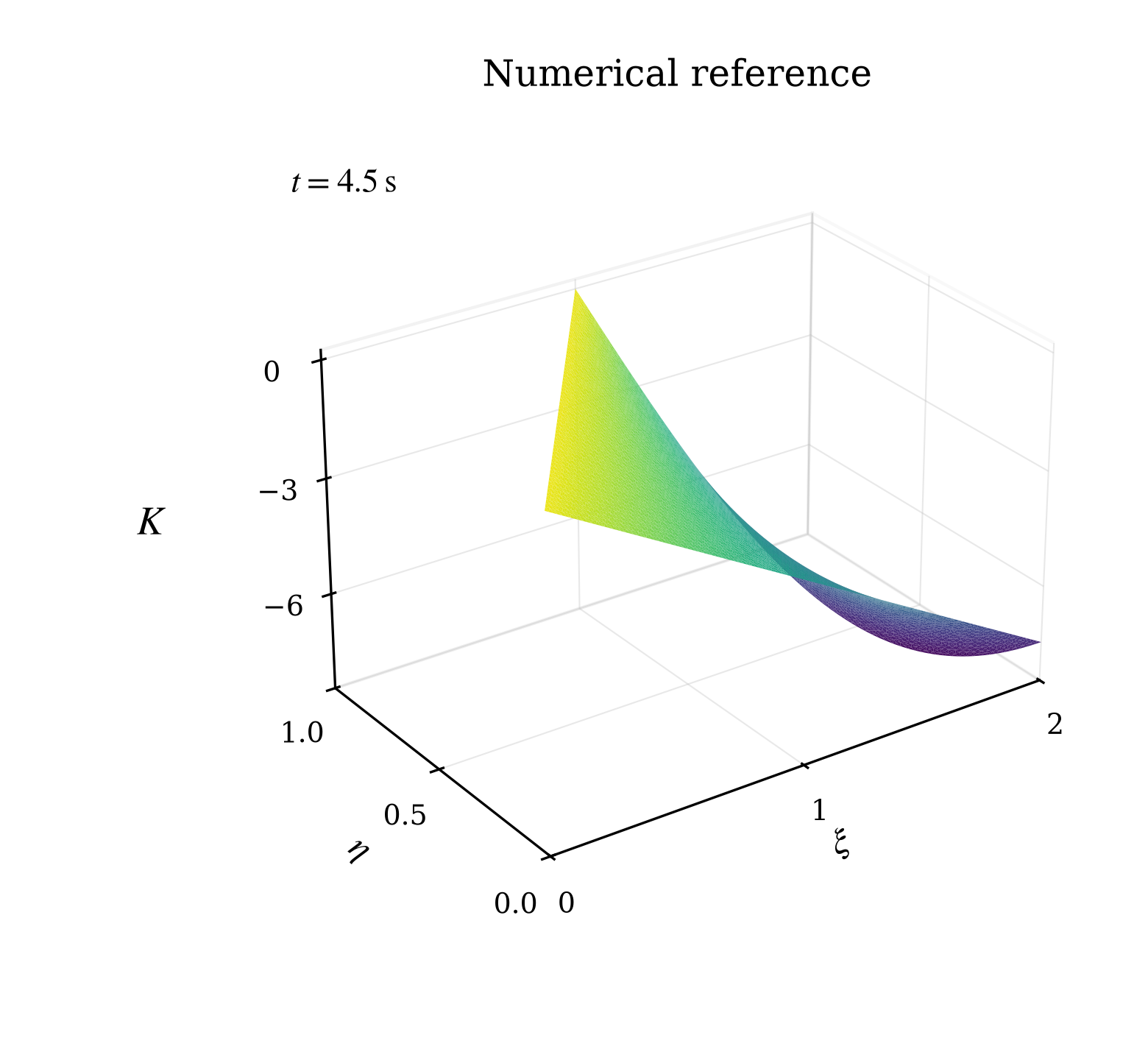} &
    \KernelPanel{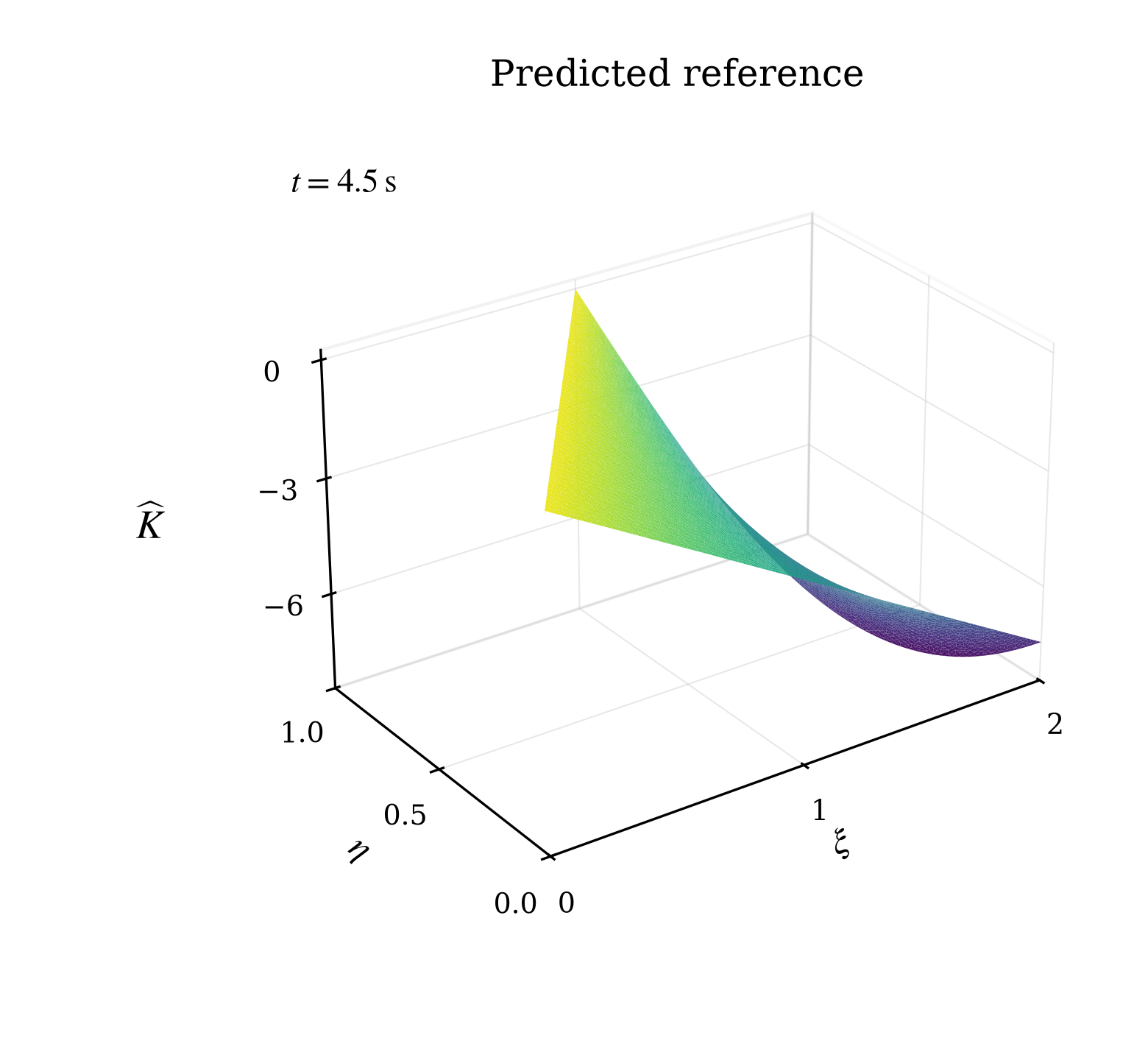} &
    \KernelPanel{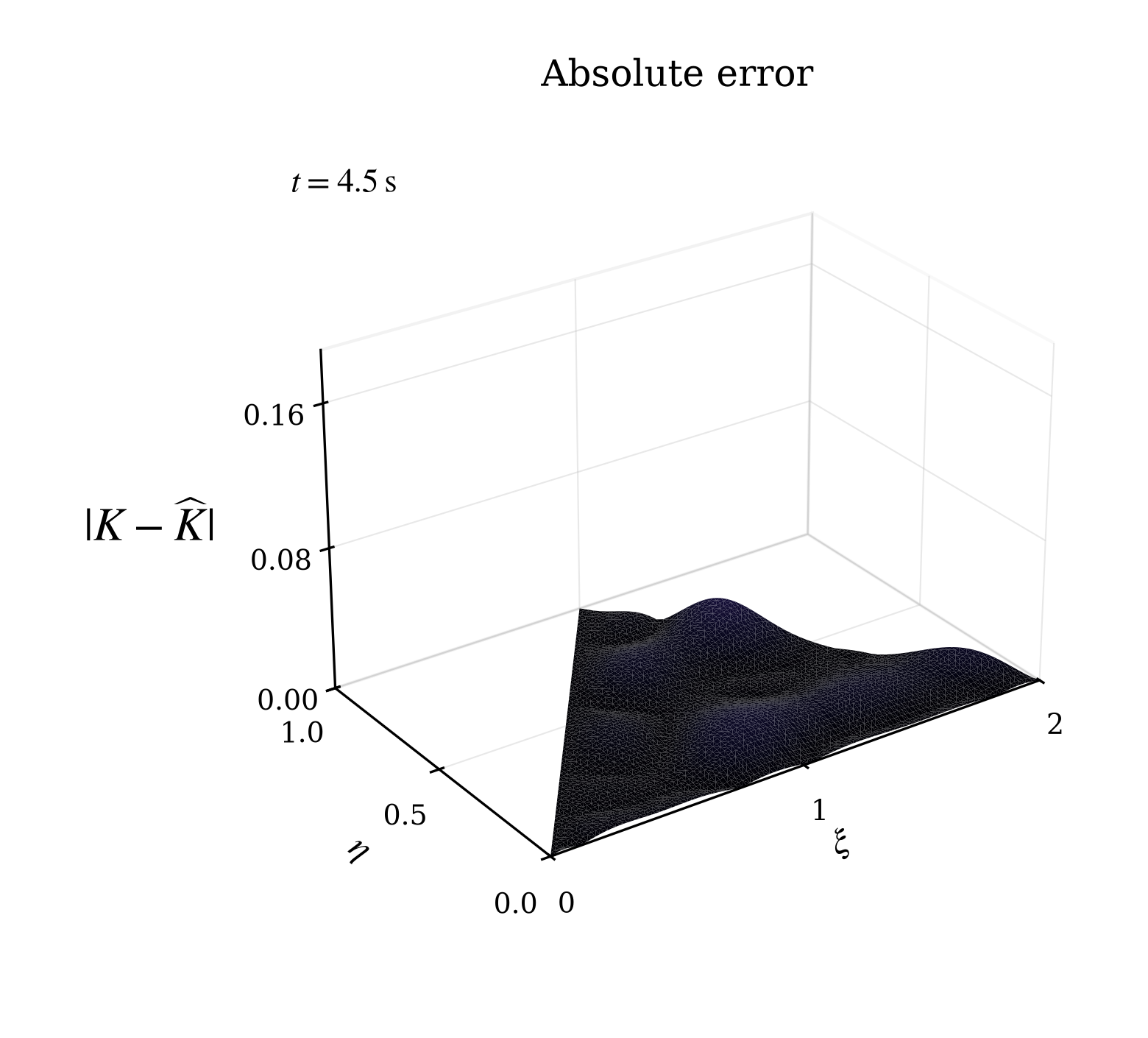}
  \end{tabular}

  \caption{Numerical kernel $K$ (left), DeepONet prediction $\widehat K$ (middle), and absolute error $|K-\widehat K|$ (right) for $P=6\,\mathrm{s}$. Rows correspond to $t=0,1.5,3,4.5\,\mathrm{s}$.}

  \label{fig:kernel_comparison}
\end{figure}

We then numerically verify that the DeepONet-based controller preserves the stabilization performance of the numerical-kernel controller. Figure~\ref{fig:state_comparison} compares the evolution of $v(x,t)$ without control (left), under the numerical-kernel controller \eqref{eq:control_original} (middle), and under the DeepONet-based controller \eqref{eq:approx_control_original} (right). While the uncontrolled state grows rapidly, both controllers drive $v(x,t)$ toward zero, exhibiting remarkable stabilization performance.


\begin{figure}[!htbp]
  \centering

  \includegraphics[
  width=0.8\linewidth,
  height=0.24\textheight,
  keepaspectratio
  ]{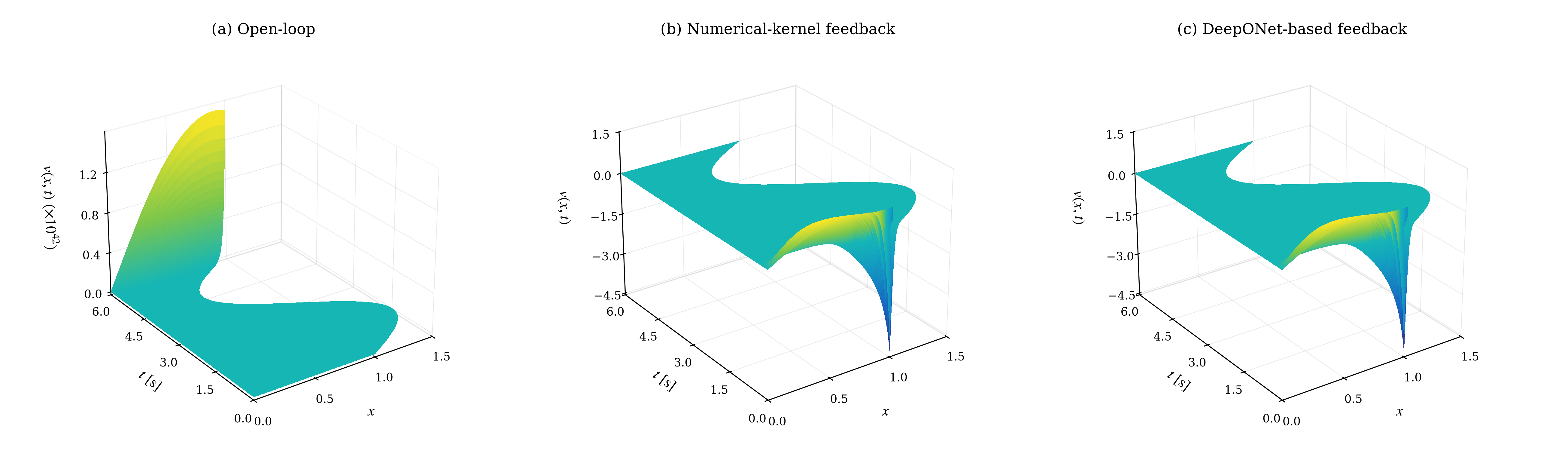}

  \caption{Evolution of $v(x,t)$: open-loop (left), numerical-kernel feedback (middle), and DeepONet-based feedback (right).}

  \label{fig:state_comparison}
  \label{fig:open_loop}
  \label{fig:closed_loop_state}
\end{figure}

Figure~\ref{fig:norm_comparison} compares the $L^2$ norms of the uncontrolled state (gray solid line) and the closed-loop states under the numerical-kernel controller \eqref{eq:control_original} (blue dashed line) and the DeepONet-based controller \eqref{eq:approx_control_original} (orange solid line) on a logarithmic scale. While the uncontrolled norm grows rapidly, both controlled norms converge rapidly to zero, exhibiting nearly identical decay behavior.

\begin{figure}[!htbp]
  \centering

  \includegraphics[
  width=0.4\linewidth,
  height=0.2\textheight,
  keepaspectratio
  ]{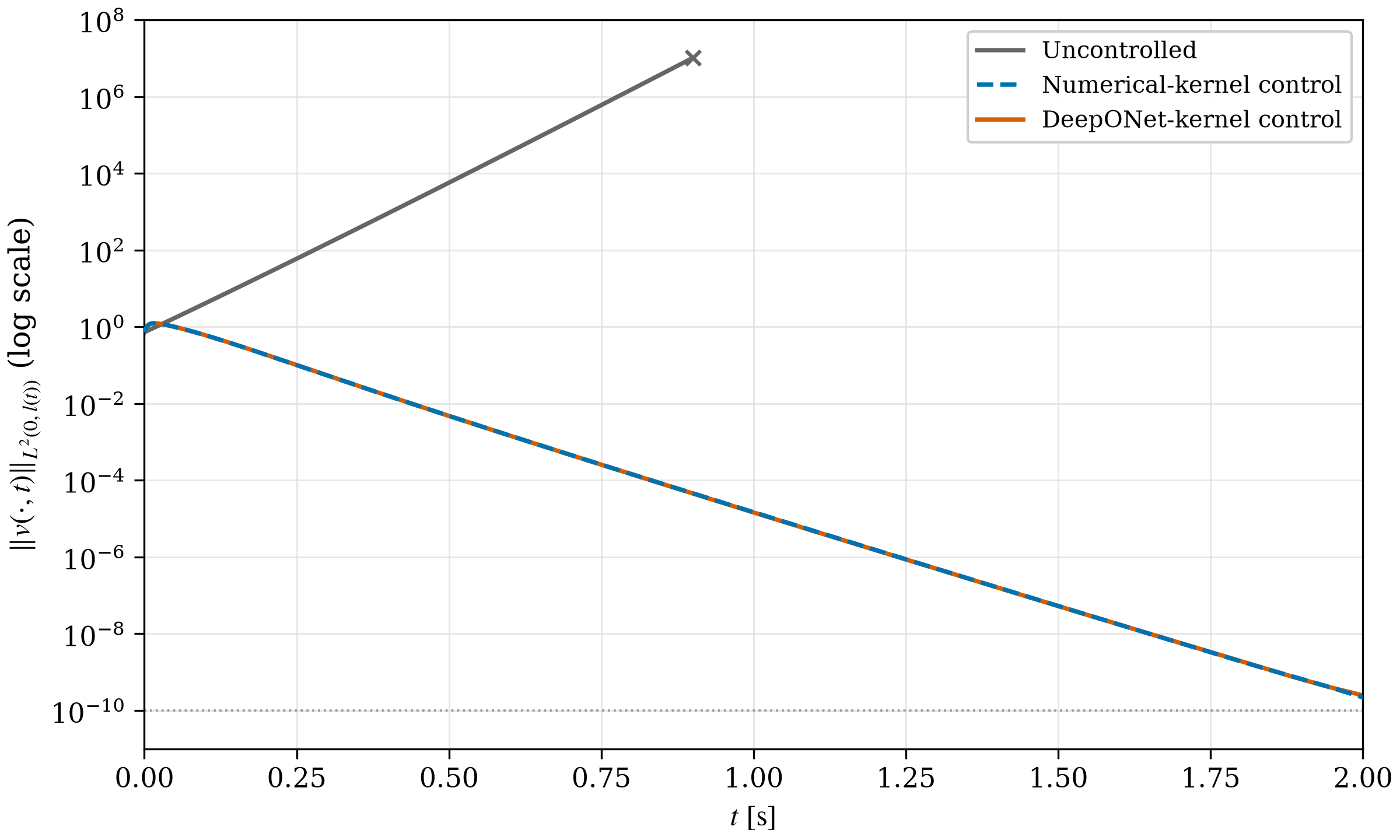}
  \caption{Uncontrolled and controlled $L^2$ norms on a logarithmic scale.}

  \label{fig:norm_comparison}
\end{figure}

These results demonstrate that the DeepONet-based approach significantly reduces the computational time required to obtain the feedback kernels while preserving the closed-loop stabilization performance. This indicates that the proposed approach can substantially improve computational efficiency without compromising the stabilization of the system, even for previously unseen trajectories of the moving boundaries.

\section{Conclusion}
\label{sec:conclusion}
In this paper, we propose a NO-based framework for the boundary stabilization of systems defined on a time-varying spatial domain. We establish the existence of a NO approximation for the time-varying backstepping kernel. By proving the continuous dependence of the kernel function on the moving boundaries, we derive the existence of the approximate kernel by the universal approximation theorem. Furthermore, we prove that the exact kernel inherits the periodicity of the moving boundary. The modified DeepONet ensures that the approximate kernel is also periodic. Moreover, uniform exponential stability is established under a less restrictive approximation accuracy requirement than that for the unmodified DeepONet. From a computational perspective, our approach achieves a substantial speedup over the conventional numerical method.


Future research will be directed toward extending this operator-learning framework to multi-dimensional moving-boundary problems. Additionally, it is an interesting direction to extend this framework to tracking problems, see \cite{Yao2026Asymptotical}.
\section*{Funding}
This work was supported by National Natural Science Foundation of China NSFC-62473281, 12571484.

\bibliographystyle{plain}
\bibliography {myref.bib}
\end{document}